\documentclass{article}
\usepackage[utf8]{inputenc}
\usepackage{amsthm}
\usepackage{amsmath}
\usepackage{graphicx}
\usepackage{hyperref}
\usepackage{mathtools}
\usepackage{amssymb}
\usepackage{physics}
\usepackage{appendix}
\graphicspath{ {./images/} }
\usepackage[margin=1.0in]{geometry}
\usepackage{comment}
\usepackage{xcolor}
\usepackage{thmtools}
\usepackage{thm-restate}
\usepackage{multicol}
\usepackage{enumitem}
\usepackage{array}

\usepackage{booktabs}
\usepackage{nicefrac}
\usepackage{tikz}
\usetikzlibrary{patterns}

\newtheorem{lemma}{Lemma}

\newtheorem{cor}{Corollary}
\newtheorem{conj}{Conjecture}
\newtheorem{definition}{Definition}

\usepackage{algorithm}
\usepackage[noend]{algpseudocode}

\newcommand{\recurfrac}{\bigg[\frac{n}{k_{s}+k_{r}}\bigg]}

\makeatletter
\newcommand{\subalign}[1]{%
  \vcenter{%
    \Let@ \restore@math@cr \default@tag
    \baselineskip\fontdimen10 \scriptfont\tw@
    \advance\baselineskip\fontdimen12 \scriptfont\tw@
    \lineskip\thr@@\fontdimen8 \scriptfont\thr@@
    \lineskiplimit\lineskip
    \ialign{\hfil$\m@th\scriptstyle##$&$\m@th\scriptstyle{}##$\hfil\crcr
      #1\crcr
    }%
  }%
}
\makeatother

\title{On the Number of Distinct Tilings of Finite Subsets of $\mathbb{Z}^{d}$ With Tiles of Fixed Size}
\author{Jesse Stern}
\begin{document}

%\pagecolor{back}
%\color{white}

\maketitle

\begin{abstract}
In this work, we study the number of finite tiles $A\subset\mathbb{Z}^{d}$ of size $\alpha$ that can translationally tile a finite $C\subset\mathbb{Z}^{d}$. Under these restrictions the tile $A$ can be translated any number of times to cover exactly $C$, but cannot be rotated or reflected. Further we consider two tiles $A$ and $A'$ to be congruent (i.e. not distinct) if and only if one can be transformed into the other via some translation. We make several significant contributions to the study of this problem. For any $\alpha\in\mathbb{Z}^{+}$ and $C=[x_{1}]\times[x_{2}]\times\ldots [x_{d}]$ where $x_{1},\ldots,x_{d}\in\mathbb{Z}^{+}$ (which we refer to as a finite contiguous $C$), we classify exactly which $A$ of size $\alpha$ can tile $C$. More specifically, we give an efficient\footnote{By efficient, we mean polynomial time with respect to $|\mathcal{T}(\alpha,C)|$.} method for enumerating all elements of $\mathcal{T}(\alpha,C)$, where $(A,B)\in \mathcal{T}(\alpha,C)$ if and only if
\begin{multicols}{2}
\begin{enumerate}[itemsep=0mm]
\item $A,B\subset\mathbb{Z}^{d}$
\item $A+B=C$
\item $|A|=\alpha$
\item $|C|=\alpha|B|$,
\end{enumerate}
\end{multicols}
\noindent where we use $A+B$ to mean the Minkowsji sum of $A$ and $B$. Further, we assume $(A,B)$ is some canonical representative (to be formally defined later) of the class of all $(A',B')$ such that $A$ is congruent to $A'$. This classification of the elements of $\mathcal{T}(\alpha,C)$ also allows us to prove a partial order on $|\mathcal{T}(\alpha,C)|$ with respect to $\alpha$ for any finite contiguous $C$.

We then study the extremal question as to the the growth rate of $\text{max}_{\alpha,C}[|\mathcal{T}(\alpha,C)|]$ with respect to $|C|$. The trivial bounds for this value, even for restricted classes of $C$ (such as $C=[n]$ for $d=1$), are very poor, with a lower bound of roughly $\log n$ and an upper bound of ${n\choose n/2}$. We improve these bounds for finite contiguous $C$ to the upper bound of
\[n^{\frac{(1+\epsilon)\log n}{\log\log n}}\]
and an infinitely often super-polynomial lower bound. More specifically, the lower bound states that, for all constants $c$ and some infinite $N\subset\mathbb{Z}^{+}$,
\[\forall n\in N, \exists\alpha\in\mathbb{Z}^{+}(|\mathcal{T}(\alpha,C)|>n^{c}),\]
where $n=|C|$ and $C$ is both finite and contiguous. 

We conjecture that the number of tilings of any finite contiguous $C$ by tiles of size $\alpha$ is an upper bound on the number of tilings of any finite $C'\subset \mathbb{Z}^{d}$ by tiles of size $\alpha$. In an effort to begin working towards this and other results for non-contiguous $C$, we prove that any $A$ of size $\alpha$ that tiles some finite contiguous $C$ itself has at most as many tilings by tiles of size $\alpha'$ (for any $\alpha'\in\mathbb{Z}^{+}$) as there are tilings of $[\alpha]$ by tiles of size $\alpha'$.
\end{abstract}

\section{Introduction}
Translational tilings of $\mathbb{Z}$ \cite{CM99,LW96,Tijdeman93}, and of $\mathbb{Z}^{d}$ \cite{BKT23,GT20,KM09,Szegedy98} more generally, are natural problems that have been studied in many prior works. While the vast majority of work in this area seeks to tile the infinite set $\mathbb{Z}^{d}$, Pederson and Wang \cite{PW01} initiated the study of tiling finite intervals of $\mathbb{Z}$. This was later followed by the independent work of Bodini and Rivals \cite{BR06}, with the combined works outlining necessary and sufficient conditions for such tilings. Additionally, Bodini and Rival \cite{BR06}, as well as Rivals \cite{Rivals07}, initiated the study of counting the number of such tilings. In this work, we extend these results to study the number of tilings of a finite $C\subset\mathbb{Z}^{d}$ by distinct tiles of size $\alpha$, the set of which we refer to as $\mathcal{T}(\alpha,C)$. While most of our results focus on contiguous $C$ (i.e. $C$ such that $C$ is the Cartesian product of $d$ intervals), we also prove several results related to non-contiguous $C$. In summery, we expand upon the prior works in the following ways: 
\begin{enumerate}
\item Introducing new characterizations of tilings of $[n]$ that allows for the specification of a fixed tile size.
\item Proving a partial order on $|\mathcal{T}(\alpha,C)|$ with respect to $\alpha$ for any finite contiguous $C$.
\item Proving strong upper and lower bounds as to the number of such tilings.
\item Extending these results to $d$ dimensions.
\item Proving similar results for tilings of $C$ for a family of finite non-contiguous $C$.
\end{enumerate}
We now discuss how each of these contributions and the prior works that relate to them.

\subsection{Characterizations of Tilings}
With respect to the tilings of finite intervals of $\mathbb{Z}$, a number of our results align with results of Bodini and Rivals \cite{BR06} and Pederson and Wang \cite{PW01}, but allow for specification of the additional parameter $\alpha$ (i.e. a specific tile size). For us to specify properties of tilings of finite intervals of $\mathbb{Z}$ by tiles of size $\alpha$, it must be the case that any properties we highlight must also be true of tilings of finite intervals of $\mathbb{Z}$ without a specified tile size. Thus, one can recover many of the structural results of these prior works from our own by considering them for all $\alpha$. Unfortunately, there does not appear to be a straightforward way to derive our results from those prior without writing new proofs that take tile size into consideration at each step. Despite this, there are several properties of tilings utilized in our proofs that were proved in these prior works and which we summarize in Lemma \ref{SegmentRiftSize}. For completeness, we provide a proof Lemma \ref{SegmentRiftSize} in the terminology of this paper.

\subsection{Counts on the Number of Distinct Tilings}
As to the counting of tilings of finite subsets of $\mathbb{Z}^{d}$, Bodini and Rivals \cite{BR06} and Rivals \cite{Rivals07} began working towards this by counting the tilings of finite intervals of the discrete line (i.e. the case of $d=1$). While the recent work of Benjamini, Kozma, and Tzalik counts the number of tiles from some finite contiguous subset of $\mathbb{Z}^{d}$, the tiles counted are those that tile the infinite set $\mathbb{Z}^{d}$ as a whole, which is a fundamentally different question then tiling a finite subset of $\mathbb{Z}^{d}$. Given this, we focus on the results of Bodini and Rivals \cite{BR06} and Rivals \cite{Rivals07}, who begin by proving that the number of tilings of a finite interval of $\mathbb{Z}$ is equal to the elements of the integer sequences A067824 and A107067 as indexed by The On-Line Encyclopedia of Integer Sequences \cite{Sloane64}, though we note that the equivalence of sequences A067824 and A107067 was independently established by Karhumaki, Lifshits, and Rytter \cite{KLR10}. Unfortunately, neither sequence has clear upper or lower bounds. Further, our goals differ somewhat from those of Rivals as we wish to study fixed tile sizes and extend beyond finite intervals to some finite non-contiguous $C$. Thus, even if one did derive a satisfactory upper or lower bound from these integer sequences, this would not provide an upper of lower bound on the number of tilings for any particular fixed tiles size. This means that we require new formulas for counting tilings, even in the case of finite intervals of $\mathbb{Z}$. While we need a new sequence and different analysis for our desired setting and bounds, their results are still highly relevant to this work. The integer sequence A107067 was originally established as a count on the number of polynomials with coefficients in $\{0,1\}$ that divide $x^{n-1}$. As there are sub-linearly many distinct tile sizes that can tile a finite interval of the discrete line and because our upper and lower bound are not sensitive to such small multiplicative factors, our upper and lower bounds also serve as bounds on the number of polynomials with coefficients in $\{0,1\}$ that divide $x^{n-1}$ due to these past results.

\subsection{Upper and Lower Bounds on the Number of Distinct Tilings}
While we are unaware of any previous literature counting the number of tilings of $[n]$ with tiles of a fixed tile size $\alpha$, one can derive some trivial upper and lower bounds on the number of such tilings. For an upper bound on $\mathcal{T}(\alpha,[n])$ and without relying on prior work, one can simply note that there are at most ${n-1\choose\alpha-1}$ ways to make a tile of size $\alpha$ using elements in $[n]$ and such that the tile always includes $1$. While many of these tiles would fail to tile $[n]$, by setting $\alpha$ to $n/2$ we get a very crude upper bound of ${n\choose n/2}$. Another more in-depth approach would be to use our arguments from Lemma \ref{DivisorsToTilings} and Corollary \ref{UpperBoundCalculation} along with Theorem 5 of Bodini and Rivals \cite{BR06}, but this would result in the upper bound we obtain raised to the $\log n$.\footnote{This is due to not knowing how the distinct tilings are distributed with respect to tile size, necessitating a multiplication by the number of distinct tile sizes (which itself equals the number of divisors of the size of the set to be tiled).} To shave of this additional $\log n$ multiplicative factor from the exponent, we are thus motivated to give Definition \ref{RecursiveDef} and prove Lemma \ref{TightBound}.

As to lower bounds for $\mathcal{T}(\alpha,[n])$, we note that a trivial lower bound of $\log n$ can be obtained without prior work by considering tilings of $[2^{k}]$ for tiles of size $2^{k/2}$. Despite the fact that we are seeking a lower bound with respect to a fixed tile size, it would be theoretically possible to use results about the properties of tilings of any size to achieve this. If one could show that the number of tilings of $C$ across all possible tile sizes were super-polynomial with respect to $|C|$, one could use the fact that the number of distinct tile sizes is small (i.e. sub-linear in $n$) to prove that there exists some tile size $\alpha$ such that there are super-polynomial distinct tiles of size $\alpha$ able to tile $C$. On the one hand, one could interpret the prior observation as highlighting that one need not consider the restriction of a fixed tile size to achieve a super-polynomial lower bound in the case of a fixed tile size. That being said, we feel a more appropriate takeaway is that it is unlikely that it would be substantially easier to prove a super-polynomial lower in the case of no fixed tile size than the proof we use here for fixed tile size (despite the former set of tilings technically being larger). Thus, the fact that our argument only allows tilings of a fixed size likely causes no additional complexity in the proof and makes the end result of Lemma \ref{LowerBoundCalculation} more transparent by highlighting the particular family of sets and tile sizes driving the lower bound.
Before we are able to proceed with the lower bound calculation however, we require a new formula for enumerating the elements of $\mathcal{T}(\alpha,[n])$. The way we chose to enumerate the elements of $\mathcal{T}(\alpha,[n])$ in Lemma \ref{TightBound} does not appear conducive to good lower bound arguments. Thus, in Lemma \ref{TightBound2}, we establish a second formula for enumerating the elements of $\mathcal{T}(\alpha,[n])$ utilizing the inclusion-exclusion principle. From this, we are able to achieve our super-polynomial lower bound on $|\mathcal{T}(\alpha,[n])|$ for infinitely many $n$ and certain $\alpha$ chosen based upon $n$. We note that such a lower bound cannot be achieved for all $n$ nor for all $\alpha$, as if $n$ is prime or $\alpha=1$ there is at most a single tiling of $[n]$. 

\subsection{Generalizations to Higher Dimension and Tiling Non-Contiguous Sets}
While our aforementioned results consider more restricted sets of tilings (by allowing for an additional restriction in the form of a fixed tile size), we also work to reduce the number of restrictions on the underlying set to be tiled. We do this by considering tilings in higher dimensions as well as tilings of sets with non-contiguous elements. In their conclusion, Bodini and Rivals \cite{BR06} claim that many of their results can be extended to tilings of a $d$-dimensional rectangle and give a brief high-level description as to the technique for doing so. We prove that our results with respect to fixed tile size can be extended to $d$-dimensions. For the structural results of this form, these proceed in essentially the manner Bodini and Rivals \cite{BR06} describe (though again, with the additional need to take into account fixed tile size). As one can remove the fixed tile size restriction in these results by summing over all $\alpha$, these results also confirm the claim of Bodini and Rivals \cite{BR06}. In addition we note that our upper and lower bounds from the $1$-dimensional case still apply in $d$-dimensions. As to our counting results, their extension to $d$-dimensions follows readily from the generalization of the structural results to $d$-dimensions.

As to the tiling of finite non-contiguous $C$ we prove an upper bound on the number of such tilings with a tile of size $\alpha$ for a particular family of such $C$. More specifically, we prove that any $A$ of size $\alpha$ that tiles $C=[x_{1}]\times[x_{2}]\times\ldots [x_{d}]$ has at most as many tilings of size $\alpha'$ (for any $\alpha'\in\mathbb{Z}^{+}$) as there are tilings of $[\alpha]$ by tiles of size alpha (i.e. $|\mathcal{T}(\alpha',A)|\leq|\mathcal{T}(\alpha',[\alpha])|$).

\subsection{Prior Work On Tilings of Some Finite Non-Contiguous Multisets}
While there does not appear to be prior work on the number of tilings of non-contiguous finite subsets of integers, some research on the Turnpike Problem \cite{LSS03} (also referred to as the Partial Digest problem \cite{SS94}) implicitly studies the number of ``tilings" of a restricted familiy of finite multisets of integers. By the tiling of a multiset, we mean that the elements of the multiset are covered by translates of a given tile a number of times exactly equal to their multiplicity in the multiset. The turnpike problem asks, given a multiset of ${n\choose 2}$ integers, if there exists a set of $n$ points on a line with exactly this multiset of interpoint distances. While never phrased as a tiling problem in prior work, we note that it is equivalent to the following multiset tiling problem:
\begin{enumerate}
\item Take the given multiset $M$ of ${n\choose 2}$ integers (counting multiplicity), and build $M'$ by including in it the elements of $M$, the elements of $-1\cdot M$, and $0$ with multiplicity $n$.
\item Ask if there exists a tile of size $n$ that tiles $M'$.
\end{enumerate}
It is straightforward to prove that such tilings of $M'$ are in one-to-one correspondence with valid sets of $n$ points on a line whose set ${n\choose 2}$ interpoint distances equal $M$. Thus, a count on the number of valid turnpike reconstructions is also a bound on the number of distinct tiles that can tile $M'$ of this highly specific form. Lemke, Skiena, and Smith \cite{LSS03} give a polynomial\footnote{The exact degree of the polynomial bound depends on if the underlying point set is permitted to have points with multiplicity greater than 1, though they handle both cases.} upper bound on the number of solutions to the turnpike problem. They do this by transforming the problem into a question about the factorization of related polynomials, then counting the number of irreducible factors that are not self-reciprocal. This is interesting, as this polynomial formulation of the problem has a similar flavor to characterizations of tilings of finite intervals studied in the algebraic approach of Bodini and Rivals \cite{BR06} and Theorem 2.5 of Pederson and Wang \cite{PW01}. While we did not use this approach in our method of counting, we leave as an open question whether or not some of the techniques from either body of work can be applied to the other to yield improved results.

\section{Definitions}
For $n\subset\mathbb{Z}^{+}$ and $x,y\subset\mathbb{Z}$, we let $[n]\triangleq\{1,\dots,n\}$ and $[x,y]\triangleq\{x, x+1, x+2,\dots,y-1, y\}$ where $x\leq y$. We take $p_{i}$ to be the $i^{\text{th}}$ prime and we use $x|y$ to mean $x$ divides $y$. Let the divisor function $\sigma_{0}(n)$ for $n\in \mathbb{Z}$ equal the number of positive divisors of $n$. We write $\omega^{*}(n)$ and $\Omega^{*}(n)$ to  the prime omega functions, where $\omega^{*}(n)$ equals the number of distinct prime factors of $n$ (ignoring multiplicity of these prime factors) and $\Omega^{*}(n)$ equals the total number of prime factors of $n$ (i.e. the sum of the exponents across all prime factors of $n$).\footnote{The prime omega functions are usually denoted by $\omega(n)$ and $\Omega(n)$ for ignoring and counting multiplicity of the prime factors respectively, however we use $\omega^{*}(n)$ and $\Omega^{*}(n)$ respectively so as to avoid confusion with asymptotic notation.} Throughout this work we use $\log$ to mean $\log_{2}$. All sets discussed in this work are assumed to be finite. For sets $A$ and $B$, let $A+B=\{a+b:a\in A,b\in B\}$ be the Minkowski sum of $A$ and $B$. For a finite set $A\subset\mathbb{Z}$, we use $\text{min}[A]$ and $\text{max}[A]$ to indicate the minimum and maximum element of $A$. For a function $f$, we write $\text{min}_{x}[f(x)]$ and $\text{max}_{x}[f(x)]$ to indicate the value of $f(x)$ for any choice of $x$ the minimizes or maximizes $f(x)$ respectively.\footnote{If such an $x$ does not exist we treat this as undefined though this circumstance does not occur in this work.} When we take the max or min of a set of sets, we take the max or min respectively from the union over all elements of the set (e.g. $\text{max}\big[\big\{\{1,2\}\{5,6\}\big\}\big]=6$). Similarly, if we use a set of sets $S$ in a set difference operation, we treat $S$ as the union of its elements. Let $\text{proj}_{i}(x)$ be the projection of $x\in\mathbb{Z}^{d}$ onto the $i^{\text{th}}$ dimension (i.e. $\text{proj}_{i}(x_{1},\ldots,x_{i},\ldots,x_{d})=x_{i}$). Further, we let $\text{proj}_{i}(A)$ for a set $A$ equal the set $\{\text{proj}_{i}(x)|x\in A\}$. 
\begin{definition}\label{Contiguous}
We call a set $C$ \textbf{contiguous} if, for some $d\in\mathbb{Z}^{+}$, we have $C=[x_{1}]\times[x_{2}]\times\ldots [x_{d}]$ where $x_{1},\ldots,x_{d}\in\mathbb{Z}^{+}$.\footnote{One could alternatively define contiguous $C$ to be $C$ such that $C=[x_{1},y_{1}]\times[x_{2},y_{2}]\times\ldots [x_{d},y_{d}]$ for $x_{1},\ldots,x_{d},y_{1},\ldots,y_{d}\in\mathbb{Z}$ and $\forall i(x_{i}\leq y_{i})$, however this complicates some statements and proofs without increasing the generality of the results.}
\end{definition}
When $d=1$, a contiguous $C$ is simply a finite interval of the discrete line, though we fix its minimum point to be at $1$ for convenience.
\begin{definition}\label{TilingDef}
For finite sets $A,B,C\subset\mathbb{Z}^{d}$ such that $|A|=\alpha$ and $|B|=\beta$, the pair $(A,B)$ is a \textbf{valid translational tiling} of $C$ if and only if $A+B=C$ and $|C|=\alpha\beta$. We refer to $A$ as the \textbf{tile} and $B$ as the \textbf{translations}.
\end{definition}
We will typically drop the term ``translational" from the above definition and simply refer to $(A,B)$ as a \textbf{valid tiling} of $C$. Definition \ref{TilingDef} allows for infinitely many tilings of all finite sets as, to tile $C$, one can set $A=C+m$ and $B=C-m$ for all $m\in\mathbb{Z}^{d}$. So that we may count the number of tilings up to such translations, we give the following definition.
\begin{definition}\label{CongClass}
The \textbf{congruence class} of a tiling $(A,B)$ is the set of all tilings $(A',B')$ such that $A+m=A'$ and $B-m=B'$ for some $m\in\mathbb{Z}^{d}$ (we would also then refer to $A$ and $A'$ as themselves congruent). Let the \textbf{canonical representative} of each congruence class of tilings be $(A,B)$ such that $(0,\ldots,0)\in B$ and $\forall i(\text{min}[\text{proj}_{i}(B)]\geq 0)$.
\end{definition}
For the remainder of the paper, we presume all tilings are the canonical representative of the congruence class of tilings to which they belong. With this, we can now define the set of fixed tile size tilings of a finite set $C$.
\begin{definition}
For finite sets $A,B,C\subset\mathbb{Z}^{d}$ such that $|A|=\alpha$ and $|B|=\beta$, we define $\mathcal{T}(\alpha,C)$ to be the set of all $(A,B)$ that are the canonical representative of their congruence class of tilings, are valid tilings of $C$, and are such that $|A|=\alpha$. Further, we define $\mathcal{T}((\alpha_{1},\alpha_{2},\ldots,\alpha_{d}),C)$ to be the set of all $(A,B)$ that are the canonical representative of their congruence class of tilings, are valid tilings of $C$, and are such that $\forall{i}\in[d](|\text{proj}_{i}(A)|=\alpha_{i})$.
\end{definition}
The remaining definitions in this section are purely with respect to finite intervals of $\mathbb{Z}$ (i.e $C$ such that $d=1$). So that we may refer to a specific member of $\mathcal{T}(\alpha,C)$ when necessary, we require the following.
\begin{definition}\label{TilingOrder}
Let $C$ be a finite subset of $\mathbb{Z}$. For $T=(A,B)$ and $T'=(A',B')$ such that $T,T'\in \mathcal{T}(\alpha,C)$, we say that $T<T'$ if and only if $\text{min}[A\setminus A']<\text{min}[A'\setminus A]$. Otherwise, $T=T'$. We use $T_{i}=(A_{i},B_{i})$ to represent the $i^{\text{th}}$ valid tiling of $C$ according to this total order.
\end{definition}
To see that Definition \ref{TilingOrder} is valid, we require that it indeed defines a total order on $\mathcal{T}(\alpha,C)$ when $d=1$. To see this notice that, if $i\neq j$, then $A_{i}\neq A_{j}$, as there is a unique $B$ for tiling $C$ with translates of any fixed $A$. For the rest of the properties of a total order, we can see $A_{i}$ as corresponding to an integer that is the sum of $2^{k}$ for all $k\in A_{i}$. Thus, this being a total order follows from any subset of the integers being totally ordered by their value. We define $a_{(i,j)}$ to denote the $i^{\text{th}}$ smallest element of $A_{j}$. We will usually drop the subscript $j$ from this and other notation when it is clear from context or irrelevant (e.g. writing $a_{i}$ as opposed to $a_{(i,j)}$). We define $b_{(i,j)}$ similarly to $a_{(i,j)}$, but with reference to $B$ instead of $A$.
\begin{definition}
Let $(A_{j},B_{j})\in\mathcal{T}(\alpha,[n])$ be the $j^{\text{th}}$ valid tiling of $[n]$. We define the \textbf{first segment} and \textbf{first rift} of the $j^{\text{th}}$ tiling (i.e. $s_{(1,j)}$ and $r_{(1,j)}$ respectively) to be:
\begin{itemize}
\item $s_{(1,j)}\triangleq\{x\in A:x<\textup{min}[C\setminus A]\}$
\item $r_{(1,j)}\triangleq\{x\in C\setminus A:x<\textup{min}[A\setminus s_{1}]\}$.
\end{itemize}
As we frequently require the size of the first segment and first rift of a tiling, we let $k_{(s,j)}\triangleq |s_{(1,j)}|$ and $k_{(r,j)}\triangleq |r_{(1,j)}|$. For $i>1$, we define the \textbf{$i^{\text{th}}$ segment} and \textbf{$i^{\text{th}}$ rift} of the \textbf{$j^{\text{th}}$ tiling} (i.e. $s_{(i,j)}$  and $r_{(i,j)}$ respectively) recursively as follows:
\begin{itemize}
\item $s_{(i,j)}\triangleq\big\{x\in A:\textup{max}[s_{i-1}]<x<\textup{min}[(C\setminus A)\setminus \{y\in C:y\leq \textup{max}[r_{i-1}]\}]\big\}$
\item $r_{(i,j)}\triangleq\Bigg\{x\in C\setminus A:\textup{max}[r_{i-1}]<x<\textup{min}\bigg[A\setminus \bigg(\bigcup\limits_{k=1}^{i-1} s_{k}\bigg)\bigg]\Bigg\}$.
\end{itemize}
\end{definition}
To put the above more intuitively, the $i^{\text{th}}$ segment of a valid tiling $T_{j}$ is the $i^{\text{th}}$ set of consecutive (relative to $C$) elements of $A_{j}$ where as the the $i^{\text{th}}$ rift of a tiling $T_{j}$ is the elements of $C$ between $s_{i}$ and $s_{i+1}$. We define $S_{j}$ and $R_{j}$ to be the set of all non-empty $s_{(i,j)}$ and $r_{(i,j)}$ respectively. For example, let $A=\{1,2,5,6,9,10\}$, $B=\{0,2\}$ and $C=[12]$. Then $s_{1}=\{1,2\}$, $s_{2}=\{5,6\}$, and $s_{3}=\{9,10\}$ while $r_{1}=\{3,4\}$ and $r_{2}=\{7,8\}$. The aforementioned segments and rifts would then be exactly the elements of $S$ and $R$ respectively, as all other segments and rifts are empty in this example. We define $\mathcal{T}(\alpha,C,(k_{s},\cdot))$ to be
\[\mathcal{T}(\alpha,C,(k_{s},\cdot))\triangleq\{(A_{i},B_{i})\in \mathcal{T}(\alpha,C):|s_{(1,i)}|=k_{s}\}\]
and define $\mathcal{T}(\alpha,C,(k_{s},\cdot))$ to be
\[\mathcal{T}(\alpha,C,(k_{s},k_{r}))\triangleq\{(A_{i},B_{i})\in \mathcal{T}(\alpha,C):(|s_{(1,i)}|=k_{s})\land(|r_{(1,i)}|=k_{r})\}.\]

\section{Formulas for Enumerating $\mathcal{T}(\alpha,[n])$}
In this section, we define two formulas for counting the exact number of distinct tilings for any $\alpha$ and $C=[n]$. By summing over all tile sizes,  To prove the first of these formulas is correct, we do two things:
\begin{itemize}
\item Define necessary and sufficient conditions as to the size of segments and rifts in valid tilings.
\item Group points into \textit{meta-points} such that each element of the meta-point is in the same segment or rift as each other element of the meta-point.
\end{itemize}
Taken together, we are able to calculate $|\mathcal{T}(\alpha,[n])|$ by taking the sum of a small number of $|\mathcal{T}(\alpha',[n'])|$ for $n'<n$ where $\alpha'$ and $n'$ are straightforward to calculate from $\alpha$ and $n$. This first formula is useful for proving both our partial order on $|\mathcal{T}(\alpha,[n])|$ with respect to tile size for fixed $n$, as well as for proving our upper bound on $|\mathcal{T}(\alpha,[n])|$. Unfortunately, it is less useful for deriving a strong lower bound. Thus, we define our second formula for counting $|\mathcal{T}(\alpha,[n])|$, the correctness of which we prove from the first formula using a combinatorial argument.

Before proceeding further with our results, we present a result characterizing the tilings of discrete intervals that was independently proven by Pederson and Wang \cite{PW01} as well as by Bodini and Rivals \cite{BR06}. For completeness, we present a proof of this as Lemma \ref{SegmentRiftSize}, but with the statement of the lemma and proof in the notation of this work.

\begin{lemma}[\cite{BR06,PW01}]\label{SegmentRiftSize}
For all $\alpha,n\subset\mathbb{Z}$ and any valid tiling $(A_{i},B_{i})\in\mathcal{T}(\alpha,[n])$, all segments of the $i^{\text{th}}$ tiling have size $k_{s}$ and the length of any rift of the $i^{\text{th}}$ tiling is divisible by $k_{s}$.
\end{lemma}

\begin{proof}
Once a $k_{s}$ and $k_{r}$ have been selected and knowing that $r_{1}\neq\emptyset$, the only way to tile $r_{1}$ with translates of elements of $A$ is with translates of elements of $s_{1}$. As elements of $s_{1}$ are consecutive as are those of $r_{1}$, the only way to do this is as in Lemma \ref{LowerBound} (i.e. by defining $B^{*}$ to be $\{x\cdot k_{s}|x\in[0,k_{r}/k_{s}]\}$ and let $B^{*}$ be a subset of $B$). This also justifies the restriction of $\mathcal{R}$ that $k_{s}|k_{r}$, as otherwise tiling $r_{1}$ with translates of $s_{1}$ would not be possible (example: for $s_{1}=\{1,2\}$ and $s_{2}=\{x,x+1\}$, the number of elements in $r_{1}=[3,x-1]$ must be divisible by $2$).

Now we prove that $\forall i[(|s_{i}|\neq 0)\implies (|s_{i}|=k_{s})]$. By definition, $|s_{1}|=k_{s}$. Suppose this is the case for all $s_{j}$ such that $j<i$. Consider $s_{i}$. If $|s_{i}|=0$ the statement holds. Suppose then that $|s_{i}|\neq 0$ and that $|s_{i}|>k_{s}$. Then, as $k_{s}$ is an element of $B^{*}$, we have that $s_{i}\cap(s_{i}+k_{s})\neq\emptyset$ which is a contradiction. Suppose $|s_{i}|<k_{s}$. Let $I=\big[\text{max}[s_{i}]+1,\text{min}[s_{i}+k_{s}]-1\big]$ and observe that $1\leq|I|<k_{s}$. The lower bound on $|I|$ follows directly from $|s_{i}|<k_{s}$ and the translation of $s_{i}$ by $k_{s}$, while the upper bound on $|I|$ follows from the fact that $|s_{i}|>0$ and the fact that we are taking the max from $s_{i}$ for lower bound $I$, but we are taking the min from $s_{i}+k_{s}$ the upper bound $I$. Taken together, with the fact that $|s_{i}|>0$, we have that $|I|<k_{s}$. Notice that introducing any element to $B$ smaller then any element of $B^{*}$ would result in a collision between translates of $s_{1}$. Given this, $I$ must be tiled by some translate of $s_{j}$ for $j<i$, but $|s_{j}|=k_{s}>|r^{*}|$ and thus, we have a contradiction. Together, these prove that $\forall i[(|s_{i}|\neq 0)\implies (|s_{i}|=k_{s})]$ as desired.
\end{proof}

With this, we proceed with defining the first counting formula and proving its correctness. We do this in two main steps in which we:
\begin{itemize}
\item Prove that the restrictions to the size of the first segment and rift (i.e. $k_{s}$ and $k_{r}$ respectively) and the sub-cases we sum over based upon these these values are \textbf{sufficient} to yield a valid tiling. This proves that our formula acts as a lower bound to $|\mathcal{T}(\alpha,[n])|$.
\item Prove that the restrictions to the size of the first segment and rift (i.e. $k_{s}$ and $k_{r}$ respectively) and the sub-cases we sum over based upon these these values are \textbf{necessary} to yield a valid tiling. This proves that our formula acts as an upper bound to $|\mathcal{T}(\alpha,[n])|$.
\end{itemize}
As the value produced by our formula is both an upper and lower bound on $|\mathcal{T}(\alpha,[n])|$, it follows that it calculates the exact value of $|\mathcal{T}(\alpha,[n])|$. We can now define the first formula in full detail.
\begin{definition}\label{RecursiveDef}
For $\mathcal{S}=\{k_{s}\in\mathbb{Z}^{+}:k_{s}|\alpha\}$ and
\[\mathcal{R}_{k_{s}}=\{k_{r}\in\mathbb{Z}^{+}:(k_{s}|k_{r})\land (k_{s}+k_{r}|k_{s}\beta)\land\big((k_{s}=\alpha)\Longleftrightarrow (k_{r}=0)\big)\}\]
we define the set $\Psi_{(\alpha,[n],(k_{s},k_{r}))}$ to be
\[
\Psi_{(\alpha,[n],(k_{s},k_{r}))}\triangleq
\begin{cases}
0, & \alpha\nmid n\\
1, & k_{r}=0\\
\bigg|\mathcal{T}\bigg(\alpha/k_{s},\recurfrac\bigg)\bigg|-\bigg|\mathcal{T}\bigg(\alpha/k_{s},\recurfrac ,(1,\cdot)\bigg)\bigg|, & \textup{otherwise} 
\end{cases}
\]
\end{definition}
Using this definition, we prove that the following method can be used to count the tilings of $[n]$ by sets of size $\alpha$.
\begin{lemma}\label{TightBound}
\[|\mathcal{T}(\alpha,[n])|=\sum\limits_{k_{s}\in\mathcal{S}}\text{ }\sum\limits_{k_{r}\in\mathcal{R}_{k_{s}}}\Psi_{(\alpha,[n],(k_{s},k_{r}))}.\]
\end{lemma}
To prove Lemma \ref{TightBound}, we first prove that the right hand side is an upper bound for the left hand side. We then use this to prove equality in all cases of $\Psi_{(\alpha,[n],(k_{s},k_{r}))}$ as well as justifying the definitions of $\mathcal{S}$ and $\mathcal{R}_{k_{s}}$. After these lemma, we proceed with the formal proof of Lemma \ref{TightBound}.

We note that Lemma \ref{LowerBound} and Lemma \ref{RecurUpperBound} require a structural result similar to Theorem 4 of Bodini and Rivals \cite{BR06} and Corollary 2.2 of Pedersen and Wang \cite{PW01} in their proofs. Unfortunately, the structural result we require is distinct from these, as it requires the additional allowance for the case of fixed tile size. This necessitates a new proof which is implicitly contained in the proofs of Lemma \ref{LowerBound} and Lemma \ref{RecurUpperBound}.
\begin{lemma}\label{LowerBound}
\[|\mathcal{T}(\alpha,[n])|\geq\sum\limits_{k_{s}\in\mathcal{S}}\text{ }\sum\limits_{k_{r}\in\mathcal{R}_{k_{s}}}\Psi_{(\alpha,[n],(k_{s},k_{r}))}.\]
\end{lemma}
\begin{proof}
The definition of $\Psi_{(\alpha,[n],(k_{s},k_{r}))}$ gives us three cases, which are where its value equals either $0$, $1$, or in which its value is based on $|\mathcal{T}(\alpha',[n'])|$ (where the negative term can be seen as subtracting away the case where $k_{s}=1$). The first case, where $|\mathcal{T}(\alpha,C)|=0$, need not be handled for the lower bound, as such cases only reduce the value of the sum. For $|\mathcal{T}(\alpha,C)|=1$, as $k_{r}=0$ and $\alpha|n$, we can always set $A=[\alpha]$ and $B=\{x\cdot\alpha:x\in[0,\beta-1]\}$, resulting in $A+B=C$. Lastly, we handle the case where the value of $\Psi_{(\alpha,[n],(k_{s},k_{r}))}$ is based upon $|\mathcal{T}(\alpha',[n'])|$. To address this case, we define an injective mapping
\[f_{k_{(s,i)},k_{(r,i)}}:\mathcal{T}(\alpha/k_{(s,i)},[n/(k_{(s,i)}+k_{(r,i)})])\rightarrow \mathcal{T}(\alpha,[n],(k_{(s,i)},k_{(r,i)}))\]
where we let $T_{i}=(A_{i},B_{i})\in \mathcal{T}(\alpha,[n],(k_{(s,i)},k_{(r,i)}))$ and $T_{j}=(A_{j},B_{j})\in \mathcal{T}(\alpha/k_{(s,i)},[n/(k_{(s,i)}+k_{(r,i)})])$ and assume $k_{(s,i)},k_{(r,i)}\geq 2$. For notational convenience, we suppress the subscripts of $f$ for the remainder of this proof. We abuse notation slightly and let $f(A_{j})=A_{i}$ (or $f(B_{j})=B_{i}$) if and only if $f(T_{j})=T_{i}$. To construct such an $f$, we map $T_{j}$ to $T_{i}$ such that $m\in A_{j}+B_{j}$ if and only if $[(m-1)(k_{(s,i)}+k_{(r,i)})+1,m(k_{(s,i)}+k_{(r,i)})]\in A_{i}+B_{i}$. We call this the \textbf{\textit{key property}} of $f$. The main idea behind the key property is that, as $T_{i}$ tiles a larger set than $T_{j}$, we can expand each point of $A_{j}$ to be multiple consecutive points in $A_{i}$.\footnote{The opposite direction also holds in that on can compress consecutive points in $A_{i}$ down to single points of $A_{j}$, but this follows from the upper bound, not the lower bound}

To complete the proof, we define $f$ and show it is injective. Let $T_{j}$ be an arbitrary valid tiling in $\mathcal{T}(\alpha/k_{(s,i)},[n/(k_{(s,i)}+k_{(r,i)})])$ such that $k_{(s,j)}\geq 2$. As $T_{j}$ is arbitrary, maintaining the key property forces us to make sure that $[2(k_{(s,i)}+k_{(r,i)})]\subset A_{i}+B_{i}$ for any $A_{i}$ and $B_{i}$ such that $f_{T_{j}}=(A_{i},B_{i})$ for some $j$. Thus, let $[k_{(s,i)}]$ and $[k_{(s,i)}+k_{(r,i)}+1,2k_{(s,i)}+k_{(r,i)}]$ both be subsets of $A_{i}$. Given these facts about $A_{i}$, it follows that $B^{*}=\{x\cdot k_{(s,i)}:x\in[0,k_{(r,i)}/k_{(s,i)}]\}$ is a subset of $B_{i}$. This gives us that $[2(k_{(s,i)}+k_{(r,i)})]\subset A_{i}+B_{i}$ as desired. For other $m\in A_{j}$, we let $[(m-1)(k_{(s,i)}+k_{(r,i)})+1,(m-1)(k_{(s,i)}+k_{(r,i)})+k_{(s,i)}]$ be in $f(A_{j})=A_{i}$ to ensure that key property is maintained. To see this, notice that this implies that $B^{*}+[(m-1)(k_{(s,i)}+k_{(r,i)})+1,(m-1)(k_{(s,i)}+k_{(r,i)})+k_{(s,i)}]=[(m-1)(k_{(s,i)}+k_{(r,i)})+1,m(k_{(s,i)}+k_{(r,i)})]\subset A_{i}+B_{i}$ as desired. For all $b_{(\ell,j)}\in B_{j}$, we have $A_{j}+b_{(\ell,j)}\subset A_{j}+B_{j}$ by definition. Let $m$ be in $A_{j}+b_{(\ell,j)}$ where $b_{(\ell,j)}\neq 0$. Notice, that by adding
\[B^{*}+\bigg\{\frac{b_{\ell,j}(k_{(s,i)}+k_{(r,i)})}{k_{(s,i)}}\bigg\}\]
to $f(B_{j})=B_{i}$, we get that $[(m-1)(k_{(s,i)}+k_{(r,i)})+1,m(k_{(s,i)}+k_{(r,i)})]\subset A_{i}+B_{i}$ as desired.
\end{proof}
To give an example of the above using $A_{i}+B_{i}=[24]$ and $A_{j}+B_{j}=[48]$, the tiling $T_{(k',j)}=(\{1,3,9,11\},\{0,1,12,13\})$ would map to \[T_{(k,i)}=(\{1,2,5,6,17,18,21,22\},\{0,2,24,26\})\] via $f$. In order to prove that the sum from Lemma \ref{TightBound} (along with the given definitions for $\mathcal{S}$ and $\mathcal{R}_{k_{s}}$) act as an upper bound to the number distinct tilings of $C=[n]$, we require the following definition.
\begin{definition}\label{meta-points}
For $(A,B)\in\mathcal{T}(\alpha,C)$, let the $x^{\text{th}}$ meta-point of $C$ with respect to $A$ (denoted by $x^{*}$) be the set $[(x-1)(k_{s}+k_{r})+1,x(k_{s}+k_{r})]$. We use the terminology of segments to refer to consecutive sets of meta-points in $C$ and refer to these as \textbf{meta-segments} (i.e. $s^{*}_{i}$). We extend the idea of rifts to \textbf{meta-rifts} (i.e. $r^{*}_{i}$) similarly. More formally we have that
\begin{itemize}
\item $s^{*}_{1}\triangleq\big\{x^{*}\subset A:\textup{max}[x^{*}]<\textup{min}[C\setminus A]\big\}$
\item $r^{*}_{1}\triangleq\Big\{x^{*}\subset C\setminus A:\textup{max}[x^{*}]<\textup{min}\big[A\setminus s^{*}_{1}\big]\Big\}$.
\item $s^{*}_{i}\triangleq\bigg\{x^{*}\subset A:\textup{max}[s^{*}_{i-1}]<x<\textup{min}\Big[(C\setminus A)\setminus \big\{y\in C:y\leq \textup{max}[r^{*}_{i-1}]\big\}\Big]\bigg\}$
\item $r^{*}_{i}\triangleq\Bigg\{x^{*}\subset C\setminus A:x^{*}\subset\bigg(\textup{max}[r^{*}_{i-1}],\textup{min}\bigg[A\setminus \Big(\bigcup\limits_{k=1}^{i-1} s^{*}_{k}\Big)\bigg]\bigg)\Bigg\}$.
\end{itemize}
\end{definition}
\begin{lemma}\label{RecurUpperBound}
Suppose that $\alpha|n$, $k_{s}\in\mathcal{S}$, and $k_{r}\in\mathcal{R}_{k_{s}}\setminus\{0\}$. Then it follows that
\[\Psi_{(\alpha,[n],(k_{s},k_{r}))}=\bigg|\mathcal{T}\bigg(\alpha/k_{s},\recurfrac\bigg)\bigg|-\bigg|\mathcal{T}\bigg(\alpha/k_{s},\recurfrac ,(1,\cdot)\bigg)\bigg|=\big|\mathcal{T}(\alpha,[n],(k_{s},k_{r}))]\big|.\]
In addition, no selection of $k_{s}$ and $k_{r}$ such that $k_{s}\not\in\mathcal{S}$ and $k_{r}\not\in \mathcal{R}_{k_{s}}$ has any valid tilings associated with them. 
\end{lemma}
\begin{proof}
Lemma \ref{SegmentRiftSize} justifies the first restriction of $\mathcal{S}$, as $A$ is made up of segments, so if each segment has cardinality $k_{s}$, then it must be the case that $k_{s}|\alpha$. Further, from the proof of Lemma \ref{SegmentRiftSize}, we can conclude that $(s_{1}\cup s_{2})+B^{*}$ must tile exactly $[2(k_{s}+k_{r})]$ in any valid tiling as done in Lemma \ref{LowerBound}. For example, if $s_{1}=\{1,2\}$ and $s_{2}=\{7,8\}$, we know that $\{0,2,4\}\subset B$ as these are necessary to tile $r_{1}=[3,6]$ with translates of $s_{1}$. These elements of $B$ then also sum with the elements $s_{2}$ so that $(s_{1}\cup s_{2})+\{0,2,4\}=[12]$. Unless $C=[12]$, there are two possible ways the next elements of $C$ (i.e. $\{13, 14\}$) can be tiled. Either these elements are in $s_{3}$ or they are tiled by further translates of $s_{1}$. As we will see below, this decision for $C=[n]$ in this example ends up being akin to the choice of whether or not to include $3$ in $A$ for $C=[n/(k_{s}+k_{r})]=[n/6]$ (and with the size of $A$ reduced by a factor of $k_{s}=2$).

If $n=2(k_{s}+k_{r})$, we have found the unique valid tiling for this $k_{s}$ and $k_{r}$. In terms of meta-points, this case corresponds to tiling the set $\{1^{*},2^{*}\}$, where $1^{*}=[k_{s}+k_{r}]$ and $2^{*}=[k_{s}+k_{r}+1,2(k_{s}+k_{r}]$. Consider the case of $n=\ell(k_{s}+k_{r})$ for $\ell>2$. There are two ways to tile $2(k_{s}+k_{r})+1$ in $A+B$. Either $2(k_{s}+k_{r})+1\in s_{3}$ or $2(k_{s}+k_{r})+1\in s_{1}+b_{i}$ for some $b_{i}\in B$. It cannot be the case that $2(k_{s}+k_{r})+1\in s_{2}+b_{i}$ for some $b_{i}\in B$, as this would imply that $(s_{1}+b_{i})\cap[2(k_{s}+k_{r})]\neq \emptyset$ which would not yield a valid tiling. Suppose $2(k_{s}+k_{r})+1\in s_{3}$. It follows that $(s_{1}\cup s_{2}\cup s_{3})+B^{*}=[3(k_{s}+k_{r})]$. The number of times we repeat this process determines $|s_{i}^{*}|$, as each such decision to add $s_{i}$ for a new $i$ as soon as possible essentially adds one new point to the first meta-segment. Suppose $2(k_{s}+k_{r})+1\in s_{1}+b_{i}$. It follows that $\min[s_{3}]>4(k_{s}+k_{r})$. This is because the number of elements between $s_{1}+b_{i}$ and $s_{2}+b_{i}$ is $2(k_{s}+k_{r})-k_{s}$, but $s_{3}+B^{*}$ is a set of $2(k_{s}+k_{r})$ consecutive elements. Thus, $s_{3}$ (and by extension, $s_{3}+B^{*}$) cannot appear until at least $4(k_{s}+k_{r})+1$. Thus, this decision of how to tile $2(k_{s}+k_{r})+1$ leads to a meta-point being added to a meta-rift.

The choice between the two options outlined above as to how to tile $2(k_{s}+k_{r})+1$ is repeated once every $k_{s}+k_{r}$ elements and are the only ones, as repeating the case analysis from above leads to similar contradictions. As these potential tilings align exactly with those in Lemma \ref{LowerBound}, we know these tilings are valid. Thus, the solutions as to how to tile $[n]$ are exactly the ways to tile $[n/(k_{s}+k_{r})]$ with a valid tiling for which $k_{(s,j)}\geq 2$ (which is accounted for by the negative term with $k_{(s,j)}$ fixed to $1$). To justify the second restriction (i.e. $k_{s}+k_{r}|k_{s}\beta$) of $\mathcal{R}_{k_{s}}$ notice that, due to the fact that segments are of length $k_{s}$ and the definition of $B^{*}$, we have that $k_{s}+k_{r}$ elements are grouped into meta-points and are either tiled or not tiled as a group. Once $r_{1}$ is tiled by translates of $s_{1}$, the first $a\cdot (k_{s}+k_{r})/k_{s}$ elements of $C$ will be tiled. Thus, it must be that $(\alpha\cdot\mathcal{T}(k_{s}+k_{r})/k_{s})|\alpha\beta$, as otherwise translates of these $a\cdot (k_{s}+k_{r})/k_{s}$ elements could not tile $C$. This divisibility requirement simplifies to the restriction $(k_{s}+k_{r})|k_{s}\beta$ as required. For the last restriction to elements of $\mathcal{R}_{k_{s}}$, its necessity follows from the definition of segments and rifts.
\end{proof}
We now handle the other two cases for $\Psi_{(\alpha,[n],(k_{s},k_{r}))}$ relevant to the upper bound.
\begin{lemma}\label{NonRecurUpperBound}
For some $\mathcal{T}(\alpha,[n],(k_{s},k_{r}))$, if $\alpha\nmid n$, then $|\mathcal{T}(\alpha,[n],(k_{s},k_{r}))|=0$. Otherwise, if $k_{r}=0$, then $|\mathcal{T}(\alpha,[n],(k_{s},k_{r}))|=1$.
\end{lemma}
\begin{proof}
If $a\nmid n$, then $|C|=\alpha\beta$ is impossible and the lemma holds. If $k_{r}=0$, then $r_{1}=\emptyset$ and $A=[\alpha]$. Consider trying to change $B$ from $B=\{x\cdot \alpha:x\in[0,b-1]\}$ as defined in Lemma \ref{LowerBound}. We attempt to do this via induction on the elements of $B$. The base case would be to change $0$, but this is not allowed by the definition of $\mathcal{T}$. Suppose that $b_{i}$ is the first element that should be adjusted and assume without loss of generality that we cannot reduce it below $b_{i-1}+1$ or increase it to be greater then $b_{i+1}-1$. If we increase $b_{i}$, then $b_{i}+1$ is no longer in $A+B$ which is a contradiction. If we decrease $b_{i}$, then $A+b_{i-1}\cap A+b_{i}\neq\emptyset$. Thus, $b_{i}$ cannot be changed while still yielding a valid tiling.
\end{proof}
With the prior results of this section in hand, Lemma \ref{TightBound} immediatly follows.
\begin{proof}[Proof of Lemma \ref{TightBound}]
Lemma \ref{RecurUpperBound} bounds shows equality in the recursive case of $\Psi_{(\alpha,[n],(k_{s},k_{r}))}$ and justifies the restrictions to $\mathcal{S}$ and $\mathcal{R}_{k_{s}}$. Lemma \ref{NonRecurUpperBound} shows equality in the other two cases of $\Psi_{(\alpha,[n],(k_{s},k_{r}))}$. Summing over all valid $k_{s}$ and $k_{r}$ yields the lemma.
\end{proof}
While the formula from Lemma \ref{TightBound} is the one we use to prove our upper bound, it is not in a convenient form for the purpose of lower bound analysis. Thus, we give an alternative formula that we prove to be equivalent and which we use to prove our lower bound.
\begin{lemma}\label{TightBound2}
Let $\mathcal{P}_{(n,k)}$ be the set of products of $k$ distinct prime divisors of $n$.
\[|\mathcal{T}(\alpha,[n])|=\sum\limits_{k\in [n]}\sum\limits_{v\in\mathcal{P}_{(n,k)}}(-1)^{k+1}\bigg(|\mathcal{T}(\alpha,[n/v])|+|\mathcal{T}(\alpha/v,[n/v])|\bigg).\]
\end{lemma}
\begin{proof}
Let $k_{(s,i)}$ and $k_{(r,i)}$ be the size of the first segment and rift of the sumset tile $(A_{i},B_{i})$. We prove this by showing this formula's equivalence to the formula from Lemma \ref{TightBound}. We break the tilings from Lemma \ref{TightBound} into two case: tilings of $\mathcal{T}(\alpha,[n])$ such that $k_{s}>1$ and tilings such that $k_{s}=1$. We define a function $f_{1}$ that takes as input $(\alpha,[n])$ and an element of $\mathcal{T}(\alpha,[n/v])$, and outputs an element of $\mathcal{T}(\alpha,[n])$ such that $v|(k_{s}+k_{r})$ and $k_{s}=1$. Further we prove that, for any $(A_{i},B_{i})\in \mathcal{T}(\alpha,[n])$ such that $v|(k_{(s,i)}+k_{(r,i)})$ and $k_{(s,i)}=1$, there exists a unique $(A_{j},B_{j})\in \mathcal{T}(\alpha,[n/v])$ such that $f_{1}((\alpha,[n]),(A_{j},B_{j}))=(A_{i},B_{i})$. Similarly, we define a function $f_{2}$ that takes as input $(\alpha,[n])$ and an element of $\mathcal{T}(\alpha/v,[n/v])$, and outputs an element of $\mathcal{T}(\alpha,[n])$ such that $v|k_{s}$. Further we prove that, for any $(A_{i},B_{i})\in \mathcal{T}(\alpha,[n])$ such that $v|k_{(s,i)}$, there exists a unique $(A_{j},B_{j})\in \mathcal{T}(\alpha/v,[n/v])$ such that $f_{2}((\alpha,[n]),(A_{j},B_{j}))=(A_{i},B_{i})$.

We begin with $f_{1}$. For any $(A_{j},B_{j})\in \mathcal{T}(\alpha,[n/v])$ such that $f_{1}((\alpha,[n]),(A_{j},B_{j}))=(A_{i},B_{i})$, we let $A_{i}=\{x:x/v\in A_{j}\}$ and $B_{i}=\{y:y/v\in B_{j}\}+[0,v]$. $k_{(s,i)}=1$ follows from the fact that $v\geq 2$. As for $d|(k_{(s,i)}+k_{(r,i)})$, one can see from the definition of $B_{i}$ that the first rift must end at some multiple of $v$ which implies that $v|(k_{(s,i)}+k_{(r,i)})$. For any $(A_{i},B_{i})\in \mathcal{T}(\alpha,[n])$ such that $v|(k_{(s,i)}+k_{(r,i)})$ and $k_{(s,i)}=1$, it follows from Lemma \ref{TightBound} that the single point segments of $(A_{i},B_{i})$ occur only at positions $t$ such that $p_{i}|(t-1)$ (this follows from the division of $n$ by $k_{s}+k_{r}$ in the last case of the definition of $\Psi_{(\alpha,[n],(k_{s},k_{r}))}$). Thus, $f_{1}$ is invertable and the claim follows. As for $f_{2}$, for any $(A_{j},B_{j})\in \mathcal{T}(\alpha/v,[n/v])$ such that $f_{1}((\alpha,[n]),(A_{j},B_{j}))=(A_{i},B_{i})$, we let $A_{i}=\{x:\lfloor x/v\rfloor\in A_{j}\}$ and $B_{i}=\{y:y/v\in B_{j}\}$. $v|k_{(s,i)}$ follows from the fact that $\lfloor x/v\rfloor$ has the same value for every $v$ consecutive values of $x$. For any $(A_{i},B_{i})\in \mathcal{T}(\alpha,[n])$ such that $v|k_{(s,i)}$, it follows from Lemma \ref{TightBound} that every segment and rift is divisible by $v$. Thus, $f_{2}$ is invertable and the claim follows.

Let $\{p_{1},\ldots, p_{m}\}$ be the prime divisors of $n$. Notice that $f_{1}$ and $f_{2}$ map to disjoint subsets of $\mathcal{T}(\alpha,[n])$, but that the union of their codomains on inputs from $\mathcal{T}(\alpha,[n/p_{i}])$ and $\mathcal{T}(\alpha/p_{i},[n/p_{i}])$ respectively for all $i\in[m]$ is exactly $\mathcal{T}(\alpha,[n])$. The issue then in simply summing the size of these codomains is that an element of $\mathcal{T}(\alpha,[n/p_{i}])$ and an element of $\mathcal{T}(\alpha,[n/p_{j}])$ for $i\neq j$ may map to the same element $(A_{z},B_{z})\in \mathcal{T}(\alpha,[n])$ by $f_{1}$. By the definition of $f_{1}$, this would imply that $p_{i}p_{j}|(k_{(s,z)}+k_{(r,z)})$ and $k_{(s,z)}=1$, which means we can remove the over counting by subtracting cases for which $v$ is composed of $2$ distinct prime factors of $n$ (though now we may be under counting). More generally, we can apply the inclusion-exclusion principle with respect to the number of prime factors of $v$ to arrive at an exact count as desired. 
\end{proof}

\section{Upper and Lower Bound Calculations}
In this section, we have three primary results: 
\begin{enumerate}
\item A partial order on $|\mathcal{T}(\alpha,[n])|$ with respect to $\alpha$.
\item A upper bound on $|\mathcal{T}(\alpha,[n])|$ for sufficiently large $n$ utilizing $\Psi_{(\alpha,[n],(k_{s},k_{r}))}$.
\item A super-polynomial lower bound on $|\mathcal{T}(\alpha,[n])|$ for specific $\alpha$ and infinitely many $n$, proved by analyzing the formula from Lemma \ref{TightBound2}.
\end{enumerate}
Prior to beginning the proof (or series of proofs) necessary to prove each of these results, we provide a high level outline as to our approach. Further, before we dive into any of these, we wish to establish the following useful corollary.
\begin{cor}\label{abSwap}
$\mathcal{T}(\alpha,[n])=\mathcal{T}(\beta,[n])$.
\end{cor}
\begin{proof}
Let $P_{j}=(A,B)$ be a tiling of $[n]$ such that $|A|=\alpha$ and $|B|=\beta$. Define $P'_{j}$ to be $(A',B')$, where $A'=B+\{1\}$ and $B'=A-\{\text{min}[A]\}$. Notice that $P'_{j}$ is a tiling of $[n]$ such that $|A'|=\beta$ and $|B'|=\alpha$.
\end{proof}
As our results allow for fixed tile sizes, it is natural to ask, for a fixed $[n]$, are there more tiles of size $\alpha$ or of size $\alpha'$ that tile $C$? While we are unable to prove a total order on $|\mathcal{T}(\alpha,[n])|$ with respect to $\alpha$ and fully resolve this question, we are able to prove a partial order by leveraging the intuition that, the closer $\alpha$ is in its prime factorization to being a square root of $n$, the more distinct tiles there will be of size $\alpha$ that tile $[n]$. The difficulty in proving this statement is in formalizing this notion of ``closeness" to being a square root of $n$. Our proof simplifies this question by proving an order with respect to $\alpha$ and $\alpha'$ in cases where one has strictly higher multiplicity in each of its prime factors (and such that neither has too high of a multiplicity in any prime factor).

With this restriction as to the $\alpha$ we compare, we can utilize the recursive case of $\Psi_{(\alpha,[n],(k_{s},k_{r}))}$ to give an inductive proof and derive our desired partial order. In Lemma \ref{PartialOrder}, we first prove a base case with respect to $\alpha$ by comparing tiles of size $1$ to tiles of size $p_{i}$. We then make our inductive hypothesis with respect to both $\alpha$ and $n$, as to apply our inductive step, we will be relying on the recursion from Definition \ref{RecursiveDef}, which reduces the size of both $\alpha$ and $n$. Lastly, we pick apart the elements from the sets $\mathcal{S}$ and $\mathcal{R}_{k_{s}}$ for the cases in question, separating these pairs into groups of cases that can be more easily compared via an equality or inequality. Once all cases have been accounted for in this manner and with all inequalities being in the same direction, the proof will be complete.
\begin{lemma}\label{PartialOrder}
Let $\alpha=p_{1}^{\mu_{1}}\cdot p_{2}^{\mu_{2}}\cdot\ldots\cdot p_{k}^{\mu_{k}}$, let $n=p_{1}^{\psi_{1}}\cdot p_{2}^{\psi_{2}}\cdot\ldots\cdot p_{k}^{\psi_{k}}$, and assume $\forall i(2\mu_{i}\leq\psi_{i})$. It follows that, for all $j$ such that $1\leq\mu_{j}$, we have $|\mathcal{T}(\alpha/p_{j},[n])|<|\mathcal{T}(\alpha,[n])|$.
\end{lemma}
\begin{proof}
We proceed by induction on
\[\Omega^{*}(\alpha)=\sum_{i=1}^{k}\mu_{i}\]
and $n$, where $\Omega^{*}(\alpha)$ is the prime omega function. For $\Omega^{*}(\alpha)=1$ and all $n$, we have that $|\mathcal{T}(\alpha/p_{j},[n])|=|\mathcal{T}(1,[n])|=1$. By assumption we have $p_{j}^{2}|n$ and thus, in conjunction with Lemma \ref{TightBound}, we have that $|\mathcal{T}(p_{j},[n])|\geq|\mathcal{T}(1,[n/p_{j}])|+|\mathcal{T}(p_{j},[n/p_{j}])\geq 2$. Thus, the base case with respect to $\Omega^{*}(\alpha)$ holds for all $n$. Assuming that $|\mathcal{T}(\alpha/p_{j},[n])|<|\mathcal{T}(\alpha,[n])|$ holds for all $\alpha$ such that $\Omega^{*}(\alpha)\leq m$ and for contiguous $C$ such that $|C|<n$, we prove that this implies it holds for $\Omega^{*}(\alpha)=m+1$ and $n$. Without loss of generality, let this case be such that
\[\alpha=p_{1}^{\mu_{1}+1}\cdot p_{2}^{\mu_{2}}\cdot\ldots\cdot p_{k}^{\mu_{k}}\]
and assume that $2\mu_{1}\leq\psi_{1}$. By the assumption that $\forall i(2\mu_{i}\leq\psi_{i})$, it follows that $\mathcal{T}(\alpha,[n],(k_{s},k_{r}))$ has strictly more potential values for $k_{s}$ (i.e. has larger $|\mathcal{S}|$) when compared to $\mathcal{T}(\alpha/p_{1},[n],(k_{s},k_{r}))$, but that each value of $k_{s}\in\mathcal{S}\setminus\{\alpha,\alpha/p_{1}\}$ has exactly one more option for $k_{r}$ (i.e. $|\mathcal{R}_{k_{s}}|$ is one larger in the case of $\alpha/p_{1}$). This latter fact follows from the requirement that $k_{s}+k_{r}|k_{s}\beta$ and that the case with $\alpha/p_{1}$ results in an associated $\beta'$ such that $\beta'=p_{1}\beta$, as we have $(\alpha/p_{1})(\beta')=n$ by the definition of a valid tiling. Thus, $k_{s}+k_{r}=p_{1}k_{s}\beta$ is possible in the case of $\mathcal{T}(\alpha/p_{i},[n],(k_{s},k_{r}))$, but not $\mathcal{T}(\alpha,[n],(k_{s},k_{r}))$. The only exception to this is when we have $\mathcal{T}(\alpha,[n],(\alpha,k_{r}))$ and $\mathcal{T}(\alpha/p_{1},[n],(\alpha/p_{1},k_{r}))$ as in both such cases the only possible $k_{r}$ is $0$.

Based on the above, we begin by separating out the easiest cases to compare for both $\mathcal{T}(\alpha,[n])$ and $\mathcal{T}(\alpha/p_{1},[n])$, then proceed with handling the outliers. First, we have that
\[|\mathcal{T}(\alpha,[n],(\alpha,0))|=|\mathcal{T}(\alpha/p_{1},[n],(\alpha/p_{1},0))|=1.\]
Next, for $k_{s}\in\mathcal{S}\setminus\{\alpha,\alpha/p_{1}\}$ and $k_{r}$ such that $k_{s}+k_{r}|k_{s}\beta$, we have that 
\[|\mathcal{T}(\alpha,[n],(k_{s},k_{r}))|=\bigg|\mathcal{T}\bigg(\alpha/k_{s},\bigg[\frac{n}{k_{s}+k_{r}}\bigg]\bigg)\bigg|-\bigg|\mathcal{T}\bigg(\alpha/k_{s},\bigg[\frac{n}{k_{s}+k_{r}}\bigg],(1,\cdot)\bigg)\bigg|\]
as well as
\[|\mathcal{T}(\alpha/p_{1},[n],(k_{s},k_{r}))|=\bigg|\mathcal{T}\bigg(\alpha/p_{1}k_{s},\bigg[\frac{n}{k_{s}+k_{r}}\bigg]\bigg)\bigg|-\bigg|\mathcal{T}\bigg(\alpha/p_{1}k_{s},\bigg[\frac{n}{k_{s}+k_{r}}\bigg],(1,\cdot)\bigg)\bigg|.\]
By our inductive hypothesis, we have that
\[\bigg|\mathcal{T}\bigg(\alpha/k_{s},\bigg[\frac{n}{k_{s}+k_{r}}\bigg]\bigg)\bigg|>\bigg|\mathcal{T}\bigg(\alpha/p_{1}k_{s},\bigg[\frac{n}{k_{s}+k_{r}}\bigg]\bigg)\bigg|.\]
By Lemma \ref{TightBound} and our prior observation about the relative number of $k_{r}$ in such cases, we have that
\[\bigg|\mathcal{T}\bigg(\alpha/k_{s},\bigg[\frac{n}{k_{s}+k_{r}}\bigg],(1,\cdot)\bigg)\bigg|=\bigg|\mathcal{T}\bigg(\alpha/p_{1}k_{s},\bigg[\frac{n}{k_{s}+k_{r}}\bigg],(1,\cdot)\bigg)\bigg|+1.\]
Thus, it follows that $|\mathcal{T}(\alpha,[n],(k_{s},k_{r}))|\geq |\mathcal{T}(\alpha/p_{1},[n],(k_{s},k_{r}))|$ for all such $k_{s}$ and $k_{r}$.

We now classify all remaining tilings for tiles of size $\alpha/p_{1}$. In this case we have $\mathcal{T}(\alpha/p_{1},[n],(k_{s},k_{r}))$ for $k_{s}\in\mathcal{S}\setminus\{\alpha/p_{1}\}$ and $k_{r}$ such that $k_{s}+k_{r}=p_{1}k_{s}\beta$. As $n=\alpha\beta$ by definition, we also have that
\[\frac{n}{k_{s}+k_{r}}=\frac{\alpha\beta}{p_{1}k_{s}\beta}=\frac{\alpha}{p_{1}k_{s}}\]
from which it immediately follows that
\[|\mathcal{T}(\alpha/p_{1},[n],(k_{s},(p_{1}\beta-1)k_{s}))|=|\mathcal{T}(\alpha/p_{1}k_{s},[\alpha/p_{1}k_{s}])|-|\mathcal{T}(\alpha/p_{1}k_{s},[\alpha/p_{1}k_{s}],(1,(\cdot)|=1.\]
Further, this implies that $\sum_{k_{s}}|\mathcal{T}(\alpha/p_{1},[n],(k_{s},(p_{1}\beta-1)k_{s}))|=|\mathcal{S}\setminus\{\alpha,\alpha/p_{1}\}|$ for $k_{s}\in\mathcal{S}\setminus\{\alpha,\alpha/p_{1}\}$.

Lastly, we show that there are a strictly greater then $|\mathcal{S}\setminus\{\alpha,\alpha/p_{1}\}|$ tilings in $\mathcal{T}(\alpha,[n])$ we have not yet counted, thus proving the lemma. Consider the tilings in $\mathcal{T}(\alpha,[n],(\alpha/p_{1},k_{r}))$. Notice that, by our assumption that $\forall i(2\mu_{i}\leq\psi_{i})$ and with $k_{s}=\alpha/p_{i}$, it follows that
\[|\mathcal{R}_{\alpha/p_{i}}|>|\mathcal{S}\setminus\{\alpha,\alpha/p_{1}\}|.\]
As each $k_{r}\in\mathcal{R}_{\alpha/p_{i}}$ has at least one valid tiling associated with it, this proves the lemma.
\end{proof}
We now move on to proving our upper bound on $|\mathcal{T}(\alpha,[n])|$. Lemma \ref{DivisorsToTilings} leverages the fact that the formula from Lemma \ref{TightBound} has a recursive structure, with a branching factor based upon $|\mathcal{S}|$ and the size of $|\mathcal{R}_{k_{s}}|$ for each $k_{s}$ and at most logarithmic depth. We then use the fact that the definitions of $\mathcal{S}$ and $\mathcal{R}_{k_{s}}$ are based around divisibility to estimate the maximum number of branches based on the number of divisors of $\alpha$ and $\beta$ (and thus, the maximum number of valid tilings). We conclude the upper bound proof with Corollary \ref{UpperBoundCalculation}, which combines the result of Lemma \ref{DivisorsToTilings}, an estimate on the number of divisors of an integer, and an ideal setting of $\alpha$ to yield our upper bound. A similar proof to the one just outlined likely works to prove an upper bound on the number of tilings of an interval of the discrete line via the structural results Bodini and Rivals \cite{BR06}. However, it would introduce an additional $\sigma_{0}(n)$ multiplicative factor to Lemma \ref{DivisorsToTilings}, resulting in a final upper bound that would be larger by approximately a $\log n$ multiplicative factor in the exponent.
\begin{lemma}\label{DivisorsToTilings}
$|\mathcal{T}(\alpha,[n])|\leq (\sigma_{0}(\alpha)\cdot \sigma_{0}(\beta)-\sigma_{0}(\alpha)-\sigma_{0}(\beta)+2)^{\log n}$.
\end{lemma}
\begin{proof}
By Corollary \ref{abSwap}, we can assume without loss of generality that $\alpha\geq \beta$. We prove that the number of tilings that do not violate the restrictions of Lemma \ref{LowerBound} on $\mathcal{S}$ and $\mathcal{R}_{k_{s}}$ is exactly $\sigma_{0}(\alpha)\cdot \sigma_{0}(\beta)-\sigma_{0}(\alpha)-\sigma_{0}(\beta)+2$. The number of valid choices of $k_{s}$ is exactly $\sigma_{0}(\alpha)$. As to $\mathcal{R}_{k_{s}}$, we first handle the case of $k_{s}=\alpha$. As the third restriction forces $k_{r}=0$, this results in a single valid $k_{r}$. Next, consider when $k_{s}=1$. In this case, $k_{s}|k_{r}$ for any choice of $k_{r}$, so the first restriction is satisfied. Lastly, the second restriction simplifies to $k_{r}+1|\beta$. The number of $k_{r}$ that satisfy this is exactly $\sigma_{0}(\beta)-1$, as the only divisor of $\beta$ we cannot form with the sum $k_{r}+1$ is $1$. 

We can now handle any remaining cases. Let $\textit{div}(\beta)$ be the set of divisors of $\beta$. Notice that $\textit{div}(k_{s}\beta)$ is exactly $\big(k_{s}\cdot\textit{div}(\beta)\big)\cup\textit{div}(k_{s})\cup\textit{div}(\beta)$. Due to the second restriction of $\mathcal{R}_{k_{s}}$ (i.e. $k_{s}+k_{r}|k_{s}\beta$) and that fact that $k_{s}+k_{r}>k_{s}$, we can narrow down options for $k_{r}$ from $\textit{div}(k_{s}\beta)$ to $(k_{s}\cdot\textit{div}(\beta))\cup\textit{div}(\beta)$. By the first restriction of $\mathcal{R}_{k_{s}}$, we have that $k_{s}|k_{r}$, so we can further simplify valid options for $k_{r}$ from $k_{s}\cdot\textit{div}(\beta)\cup\textit{div}(\beta)$ to $k_{s}\cdot\textit{div}(\beta)$. Lastly, $k_{r}\neq k_{s}$ and thus, the number of valid options for $k_{r}$ equals
\[|k_{s}\cdot\textit{div}(\beta)|-1=\sigma_{0}(\beta)-1.\]
Thus, for all $k_{s}$ such that $k_{s}|a$ and $k_{s}\neq a$, we have $\sigma_{0}(\beta)-1$ options for $\beta$. When $k_{s}=\alpha$, we have exactly $1$ option for $k_{s}$. Each of these $\sigma_{0}(\alpha)\cdot \sigma_{0}(\beta)-\sigma_{0}(\alpha)-\sigma_{0}(\beta)+2$ sub-cases then divides $n$ by at least $2$. Thus, if we very conservatively assume that each sub-case decreases $n$ by a factor of $2$, we get that 
\[|\mathcal{T}(\alpha,[n])|\leq (\sigma_{0}(\alpha)\cdot \sigma_{0}(\beta)-\sigma_{0}(\alpha)-\sigma_{0}(\beta)+2)^{\log n}\]
as desired.
\end{proof}
By Theorem 317 of Hardy and Wright \cite{HW79} (which is attributed to Wigert (1907)) we have that, for any $\epsilon>0$ and infinitely many sufficiently large $n$,\footnote{Only the lower bound is infinitely often. The upper bound holds for all sufficiently large $n$.}
\[2^{(1-\epsilon)\log (n)/\log\log(n)}<\sigma_{0}(n)<2^{(1+\epsilon)\log (n)/\log\log(n)}.\]
Thus, for infinitely many $n$, we have that $\sigma_{0}(n)\sim 2^{\log (n)/\log\log(n)}$.
\begin{cor}\label{UpperBoundCalculation}
For $C=[n]$, all $n$, and any $\epsilon>0$ we have that
\[\underset{\alpha}{\textup{\text{max}}}\big[|\mathcal{T}(\alpha,C)|\big]\leq (2^{\frac{(1+\epsilon)\log n}{\log\log\sqrt{n}}}-2^{\frac{\log \sqrt{n}}{\log\log\sqrt{n}}+1}+2)^{\log n}<n^{\frac{(1+\epsilon')\log n}{\log\log n}}.\]
\end{cor}
\begin{proof}
Notice that the inner $\sigma_{0}(\alpha)\cdot \sigma_{0}(\beta)-\sigma_{0}(\alpha)-\sigma_{0}(\beta)+2$ is maximized when $\alpha=\beta=\sqrt{n}$. By Lemma \ref{DivisorsToTilings}, we have at most $(\sigma_{0}(\sqrt{n})^{2}-2\sigma_{0}(\sqrt{n})+2)^{\log n}$ valid tilings. Using the bounds of Wigert \cite{HW79} yields the statement.
\end{proof}
This is a significant improvement over the trivial upper bound of ${n\choose 2}$. Next, we prove our super-polynomial lower bound on $|\mathcal{T}(\alpha,[n])|$. The primary lemmas that achieve this are Lemma \ref{TightBound2PartialLowerBound} and Lemma \ref{LowerBoundCalculation}. Rather then proving these directly, we feel that it is better to first prove several strictly weaker lemmas that only achieve sub-linear lower bounds, but which highlight the key ideas we require. From there, we add several layers to the weaker arguments to take our lower bound from sub-linear all the way to super-polynomial.

We begin our series of lower bound proofs with Lemma \ref{2kcase}. All of the lemmas related to our lower bound rely on our formula for counting $|\mathcal{T}(\alpha,[n])|$ from Lemma \ref{TightBound2}. Thus, the primary purpose of Lemma \ref{2kcase} is to highlight this by applying Lemma \ref{TightBound2} to calculate $|\mathcal{T}(2^{\lfloor k/2\rfloor},[2^{k}])|$. The choice of $n$ and $\alpha$ greatly simplify this lemma, as the choice of $n=2^{k}$ means that their are no negative terms, and our choice of $\alpha=2^{\lfloor k/2\rfloor}$ leads to an opportunity to apply Corollary \ref{abSwap}.
\begin{lemma}\label{2kcase}
For $k\in\mathbb{Z}^{+}$, we have that $|\mathcal{T}(2^{\lfloor k/2\rfloor},[2^{k}])|>2^{0.58(k-2)}$.
\end{lemma}
\begin{proof}
In this case, Lemma \ref{TightBound2} simplifies to $|\mathcal{T}(2^{\lfloor k/2\rfloor},[2^{k}])|=|\mathcal{T}(2^{\lfloor k/2\rfloor},[2^{k-1}])|+|\mathcal{T}(2^{\lfloor k/2\rfloor-1},[2^{k-1}])|$. We proceed by induction on $k$. For $k=1$, we have that $\alpha=1$ which yields a single tiling which is sufficient. If $k+1$ is even, we have 
\begin{align}
\mathcal{T}(2^{(k+1)/2},[2^{k+1}])|=&|\mathcal{T}(2^{(k+1)/2},[2^{k}])|+|\mathcal{T}(2^{(k-1)/2},[2^{k}])|\\
=&|\mathcal{T}(2^{\lceil k/2\rceil},[2^{k}])|+|\mathcal{T}(2^{(\lfloor k/2\rfloor},[2^{k}])|\\
=&2|\mathcal{T}(2^{\lfloor k/2\rfloor},[2^{k}])|,
\end{align}
where the equality between lines $2$ and $3$ is due to Corollary \ref{abSwap}. We can then apply the inductive hypothesis to yield the desired result. If $k+1$ is odd, we have that 
\begin{align*}
\mathcal{T}(2^{\lfloor (k+1)/2\rfloor},[2^{k+1}])|=&|\mathcal{T}(2^{\lfloor (k+1)/2\rfloor},[2^{k}])|+|\mathcal{T}(2^{\lfloor (k+1)/2\rfloor-1},[2^{k}])|\\
=&|\mathcal{T}(2^{k/2},[2^{k}])|+|\mathcal{T}(2^{(k/2)-1},[2^{k}])|\\>&|\mathcal{T}(2^{k/2},[2^{k}])|.\\
\end{align*}
As the number of solutions doubles every other increase of $k$ and never decreases when increasing $k$, the growth rate should be roughly $1.5^{k}$. However, due to the crudeness of our bound in the case of $k+1$ being odd, this growth is ``staggered." Thus, to make sure our lower bound does not fail for some small values of $k$ by exceeding $|\mathcal{T}(\alpha,[2^{k}])|$, we subtract $2$ from the exponent. Thus, in the case of $k=2$, we have $|\mathcal{T}(2,[4])|=2$ and $1.5^{k-2}=1$. This gives the $|\mathcal{T}(2^{\lfloor k/2\rfloor},[2^{k}])|$ a sufficient head start in size so as to not be effected by the lack of growth for odd $k+1$. Thus, it follows that $|\mathcal{T}(2^{\lfloor k/2\rfloor},[2^{k}])|>1.5^{k-2}>2^{0.58(k-2)}$. 
\end{proof}
The next two preliminary lower bound lemmas are Lemma \ref{IE2k9} and Lemma \ref{2k9case}. Lemma \ref{2k9case} extends Lemma \ref{2kcase} by multiplying $n$ by $3^{2}$. While this appears to be a minor change, it serves to illustrate the idea that increasing the number of prime factors of $n$ can have the effect of increasing the number of valid tilings by some constant factor in the exponent. However, now that we have more then one prime factor of $n$ and unlike in Lemma \ref{2k9case}, applying Lemma \ref{TightBound2} results in some positive and some negative terms. Thus, we require an inequality that allows us to remove all of the negative terms at the cost of positive terms of relativity little importance. We provide such an inequality and prove its correctness in Lemma \ref{IE2k9}, then prove Lemma \ref{2k9case} utilizing it.
\begin{lemma}\label{IE2k9}
For $k\geq 1$ and $\alpha$ such that $3|\alpha$ and $3^{2}\nmid \alpha$,
\[|\mathcal{T}(\alpha,[2^{k}3^{2}])|\geq |\mathcal{T}(\alpha,[2^{k-1}3^{2}])|+|\mathcal{T}(\alpha/2,[2^{k-1}3^{2}])|+|\mathcal{T}(\alpha/3,[2^{k}])|.\]
\end{lemma}
\begin{proof}
Our approach to this proof will be to expand $|\mathcal{T}(\alpha,[2^{k}3^{2}])|$ using Lemma \ref{TightBound2}, isolate the terms in the sum $|\mathcal{T}(\alpha,[2^{k-1}3^{2}])|+|\mathcal{T}(\alpha/2,[2^{k-1}3^{2}])|+|\mathcal{T}(\alpha,[2^{k}])|$ and prove that the sum of the remaining terms from the inclusion-exclusion formula is non-negative. By Lemma \ref{TightBound2}, we have that
\begin{align*}
|\mathcal{T}(\alpha,[2^{k}3^{2}])|=&|\mathcal{T}(\alpha,[2^{k-1}3^{2}])|+|\mathcal{T}(\alpha/2,[2^{k-1}3^{2}])|+|\mathcal{T}(\alpha,[2^{k}3])|+\\
&|\mathcal{T}(\alpha/3,[2^{k}3])|-|\mathcal{T}(\alpha,[2^{k-1}3])|-|\mathcal{T}(\alpha/6,[2^{k}3])|.
\end{align*}
We can now apply Lemma \ref{TightBound2} again, this time to the term $|\mathcal{T}(\alpha,[2^{k}3])|$ to get
\begin{align}
|\mathcal{T}(\alpha,[2^{k}3^{2}])|=&|\mathcal{T}(\alpha,[2^{k-1}3^{2}])|+|\mathcal{T}(\alpha/2,[2^{k-1}3^{2}])|+|\mathcal{T}(\alpha,[2^{k-1}3])|+|\mathcal{T}(\alpha/2,[2^{k-1}3])|+\notag\\
&|\mathcal{T}(\alpha,[2^{k}])|+|\mathcal{T}(\alpha/3,[2^{k}])|-|\mathcal{T}(\alpha,[2^{k-1}])|-|\mathcal{T}(\alpha/6,[2^{k-1}])|\\
&+|\mathcal{T}(\alpha/3,[2^{k}3])|-|\mathcal{T}(\alpha,[2^{k-1}3])|-|\mathcal{T}(\alpha/6,[2^{k}3])|\notag\\
=&|\mathcal{T}(\alpha,[2^{k-1}3^{2}])|+|\mathcal{T}(\alpha/2,[2^{k-1}3^{2}])|+|\mathcal{T}(\alpha,[2^{k-1}3])|+|\mathcal{T}(\alpha/2,[2^{k-1}3])|+\notag\\
&|\mathcal{T}(\alpha/3,[2^{k}])|-|\mathcal{T}(\alpha/6,[2^{k-1}])|+|\mathcal{T}(\alpha/3,[2^{k}3])|-\\
&|\mathcal{T}(\alpha,[2^{k-1}3])|-|\mathcal{T}(\alpha/6,[2^{k}3])|\notag\\
=&|\mathcal{T}(\alpha,[2^{k-1}3^{2}])|+|\mathcal{T}(\alpha/2,[2^{k-1}3^{2}])|+|\mathcal{T}(\alpha/2,[2^{k-1}3])|+|\mathcal{T}(\alpha/3,[2^{k}])|-\\
&|\mathcal{T}(\alpha/6,[2^{k-1}])|+|\mathcal{T}(\alpha/3,[2^{k}3])|-|\mathcal{T}(\alpha/6,[2^{k}3])|\notag
\end{align}
where equation $5$ follows from equation $4$ due to $3|a$ and $3\nmid |C|$. Thus, as $a\nmid |C|$, no valid tilings were dropped between these two lines. We now set aside the terms in $|\mathcal{T}(\alpha,[2^{k-1}3^{2}])|+|\mathcal{T}(\alpha/2,[2^{k-1}3^{2}])|+|\mathcal{T}(\alpha,[2^{k}])|$, leaving us with
\[|\mathcal{T}(\alpha/2,[2^{k-1}3])|-|\mathcal{T}(\alpha/6,[2^{k-1}])|+|\mathcal{T}(\alpha/3,[2^{k}3])|-|\mathcal{T}(\alpha/6,[2^{k}3])|.\]
Notice that
\[|\mathcal{T}(\alpha/2,[2^{k-1}3])|\geq |\mathcal{T}(\alpha/6,[2^{k-1}])|\] and
\[|\mathcal{T}(\alpha/3,[2^{k}3])|\geq |\mathcal{T}(\alpha/6,[2^{k}3])|.\]
Thus, 
\[|\mathcal{T}(\alpha/2,[2^{k-1}3])|-|\mathcal{T}(\alpha/6,[2^{k-1}])|+|\mathcal{T}(\alpha/3,[2^{k}3])|-|\mathcal{T}(\alpha/6,[2^{k}3])|\geq 0\]
as required.
\end{proof}
We now wish to leverage Lemma \ref{IE2k9} to prove an improved lower bound in the case of $|\mathcal{T}(2^{\lfloor k/2\rfloor}\cdot 3,[2^{k}3^{2}])|$. To do this, the main insight that is required is that, if one repeatedly applies Lemma \ref{IE2k9}, first to some given $\mathcal{T}(\alpha,[n])$ and then again to some of the terms that result from applying Lemma \ref{TightBound2} to $\mathcal{T}(\alpha,[n])$, there will begin to be ``collisions" between some of the resulting terms. In Lemma \ref{2k9case}, we are able to create many such collisions and utilize this high multiplicity term to improve our lower bound on $|\mathcal{T}(\alpha,[n])|$ with respect to $n$.
\begin{lemma}\label{2k9case}
For $k\in\mathbb{Z}^{+}$, for $\alpha=2^{\lfloor k/2\rfloor}\cdot 3$, we have that $|\mathcal{T}(\alpha,[2^{k}])|=\omega(2^{0.79k})$.
\end{lemma}
\begin{proof}
By Lemma \ref{IE2k9}, we know that 
\[|\mathcal{T}(\alpha,[2^{k}3^{2}])|\geq |\mathcal{T}(\alpha,[2^{k-1}3^{2}])|+|\mathcal{T}(\alpha/2,[2^{k-1}3^{2}])|+|\mathcal{T}(\alpha/3,[2^{k}])|.\]
Consider the terms on the right side of the inequality. Notice that we can also apply Lemma \ref{IE2k9} to the first two terms while leaving the third unchanged. If we continue to reapply Lemma \ref{IE2k9} to the first two of the three terms that result (keeping the third term each time), we will be left with a sum of terms in the form of Pascal's triangle.\footnote{Alternatively known as Pingala's triangle, Yang Hui's triangle, and several other names due to its repeated independent discovery.} More formally, all of the terms in the sum would be of the form 
\[{j\choose i}|\mathcal{T}(\alpha/2^{i}3,[2^{k-j}])|,\] 
where the coefficient of ${j\choose i}$ follows from the observation that, when visualized as Pascal's triangle, the term $|\mathcal{T}(\alpha/2^{i}3,[2^{k-j}])|$ appears in the $j^{\text{th}}$ row and $i^{\text{th}}$ column.

To clarify that this is indeed the case, notice that the first term on the right hand side of Lemma \ref{IE2k9} equals the term from the left hand side with $|\mathcal{C}|$ divided by two. By contrast, the second term on the right hand side of Lemma \ref{IE2k9} equals the term from the left hand side with both $\alpha$ and $|\mathcal{C}|$ divided. Thus, when we consider the number of ways to output the aforementioned third term of 
\[{j\choose i}|\mathcal{T}(\alpha/2^{i}3,[2^{k-j}])|,\] 
it follows that Lemma \ref{IE2k9} was applied $j+1$ times (in $j$ cases, $|\mathcal{C}|$ was divided by $2$, while the final application removed the $3^{2}$ term). Further, the tile size $\alpha$ was divided by $2$ only $i$ times. Thus, there are ${j\choose i}$ copies of this term when utilizing Lemma \ref{IE2k9} to fully expand $|\mathcal{T}(\alpha,[2^{k}3^{2}])|$ (except for the third right hand side terms according to Lemma \ref{IE2k9} which never expanded as mentioned prior).

We now wish to find a specific term from this expansion that balances reasonably high multiplicity with a fairly large set of tilings, as this will maximize the value of their product. This desire for a large number of distinct tilings means we wish to have fairly large $|\mathcal{C}|$, with $\alpha=\sqrt{|\mathcal{C}|}$. When $i=\lfloor k/4\rfloor$ and $j=\lfloor k/2\rfloor$, we have the term
\[{k/2 \choose k/4}|\mathcal{T}(2^{\lceil k/4\rceil},[2^{\lceil k/2\rceil}])|\]
(i.e. the term in the row $\lfloor k/2\rfloor+1$ column $\lfloor k/4\rfloor$ when seen as Pascal's triangle). We know that
\[{k/2 \choose k/4}=\Omega(2^{0.5k})\]
and, using Lemma \ref{2kcase}, we know that
\[|\mathcal{T}(2^{\lceil k/4\rceil},[2^{\lceil k/2\rceil}])|=\Omega(2^{0.58k/2})=\Omega(2^{0.29k}).\]
Taken together, this gives us at least $\Omega(2^{0.79k})$ distinct tilings. As the size of $n$ has only been adjusted by a factor of $9$, but the coefficient in the exponent has increased, this gives us $\omega(n^{0.79})$ distinct tilings.
\end{proof}
Lemma \ref{2k9case} is essentially a finite example of a more general method for improving upon Lemma \ref{2kcase}. As mentioned prior, increasing the number of distinct prime factors in $n$ in the correct manner has the effect of increasing the number of valid tilings by a constant factor in the exponent (i.e. going from at least $2^{k/2}$ tilings to $2^{ck/2}$ for some $c$). Further, as one increases the number of prime factors of $n$ in the appropriate manner, $c$ is unbounded. Thus, $\text{max}_{\alpha}[|\mathcal{T}(\alpha,[n])|]$ grows at a super-polynomial rate for infinitely many $n$. We prove this in Lemma \ref{LowerBoundCalculation}, but before we can do so we highlight two differences this proof will have when compared to Lemma \ref{2k9case}.

First, our selection of a row and column from Pascal's triangle from which to extract a term with high multiplicity must be refined, as choosing a row and column simply based upon $k$ (as we did in Lemma \ref{2k9case}) will be insufficiently precise. Second, we will require a generalized version of Lemma \ref{IE2k9}. Notice that the statement of Lemma \ref{IE2k9} is essentially a finite case (as there are only two prime factors of $n$) and thus, applying Lemma \ref{TightBound2} and matching terms in the correct way was sufficient to prove the lemma. Unfortunately, such an approach is not sufficient for the necessary generalization and the the proof of Lemma \ref{TightBound2PartialLowerBound} has much more in common with the proof of Lemma \ref{TightBound2} then Lemma \ref{2k9case}. We now proceed with the proof Lemma \ref{TightBound2PartialLowerBound}.
\begin{lemma}\label{TightBound2PartialLowerBound}
Let $p_{i}$ be the $i^{\text{th}}$ prime, and let $t$, $t_{p_{m}}$, $t_{1}$, and $t_{2}$ be defined as
\begin{multicols}{2}
\begin{description}[itemsep=0mm]
\item $t\triangleq\mathcal{T}\bigg(\alpha,\Big[2^{k}\prod_{i=2}^{m}p_{i}^{2}\Big]\bigg)$
\item $t_{p_{m}}\triangleq\mathcal{T}\bigg(\frac{\alpha}{p_{m}},\Big[2^{k}\prod_{i=2}^{m-1}p_{i}^{2}\Big]\bigg)$
\item $t_{1}\triangleq\mathcal{T}\bigg(\alpha,\Big[2^{k-1}\prod_{i=2}^{m}p_{i}^{2}\Big]\bigg)$
\item $t_{2}\triangleq\mathcal{T}\bigg(\frac{\alpha}{2},\Big[2^{k-1}\prod_{i=2}^{m}p_{i}^{2}\Big]\bigg)$
\end{description}
\end{multicols}
where we define the products to be equal to $1$ if $m=1$ and $0$ if $m=0$. It follows that
\[|t|\geq |t_{p_{m}}|+|t_{1}|+|t_{2}|.\]
\end{lemma}
\begin{proof}
Throughout this proof, we will use $(A_{t},B_{t})$ to refer to $A$ and $B$ such that $(A,B)\in t$. We prove the lemma via a combinatorial argument in which we prove that $t_{p_{m}}$, $t_{1}$ and $t_{2}$ have size equal to that of three disjoint subsets of $t$. More specifically, we define mappings $f_{p_{m}}$, $f'_{1}$, and $f'_{2}$ from $t_{p_{m}}$, $t_{1}$ and $t_{2}$ respectively to $t$, such that their codomains are disjoint. All three mappings are similar to mappings $f_{1}$ or $f_{2}$ from the proof of Lemma \ref{TightBound2}. This is especially true of $f'_{1}$ and $f'_{2}$, whereas $f_{p_{m}}$ requires a more complex piece-wise definition. After defining a mapping, we briefly analyze several properties of elements of that mappings codomain. By the fact that the elements of the codomain of each mapping have mutually-exclusive properties, we are able to conclude that the codomains of the mappings are disjoint sets as desired.

For convenience, we briefly review the definitions of $f_{1}$ and $f_{2}$ from Lemma \ref{TightBound2} at a high level. For these functions formal definitions, we refer the reader back to Lemma \ref{TightBound2}. $f_{1}(x,y)$ takes a pair as input, where $x$ is a target set of tilings $\mathcal{T}(\alpha,[n])$ and where $y$ is some $(A,B)\in\mathcal{T}(\alpha,[n/v])$ for fixed $v\geq 2$. The output of $f_{1}$ is some $(A',B')\in\mathcal{T}(\alpha,[n])$ for which $k_{s}=1$. To achieve this, when imagined as a process or algorithm, $f_{1}$ essentially ``pushes apart" the elements of $[n/v]$ by adding $v-1$ points between each (following a similar process for $A$), then builds a larger $B'$ from $B$ so as to cover $[n]$ with $A'+B'$. $f_{2}(x,y)$ takes a pair as input, where $x$ is a target set of tilings $\mathcal{T}(\alpha,[n])$ and where $y$ is some $(A,B)\in\mathcal{T}(\alpha/v,[n/v])$ for fixed $v\geq 2$. The output of $f_{2}$ is some $(A',B')\in\mathcal{T}(\alpha,[n])$ for which $k_{s}>1$. $f_{2}$ is defined in much the same as $f_{1}$, except that one adds $v$ contiguous points to $A'$ for each element of $A$.

We begin by defining $f'_{1}$ using the definition of $f_{1}$. For $(A_{t_{1}},B_{t_{1}})\in t_{1}$, let $f'_{1}(A_{t_{1}},B_{t_{1}})\triangleq f_{1}(t,(A_{t_{1}},B_{t_{1}}))$. By Lemma \ref{TightBound2} all elements in the codomain of $f_{1}$ are such that $k_{(s,t)}=1$. Further, notice that $k_{(r,t)}$ is always odd. To see this, if $k_{(s,t_{1})}>1$, then $k_{(r,t)}=1$ by a single point being inserted between the first and second points of the first segment. Thus, suppose $k_{(s,t_{1})}=1$. If $k_{(r,t_{1})}$ is even, then an odd number of points are added to the first rift to get the first rift of the tiling of $t$ and thus, $k_{r,t}$ is odd. Conversely, if $k_{(r,t_{1})}$ is odd, then an even number of points is added to the first rift to get the first rift of the tiling of $t$ and thus, $k_{(r,t)}$ is odd. Thus, in all cases, an element of the codomain of $f_{1}$ has an odd valued $k_{(r,t)}$.

Next we define $f'_{2}$ using the definition of $f_{2}$ from Lemma \ref{TightBound2}. Let $f'_{2}(A_{t_{2}},B_{t_{2}})=f_{2}(t,(A_{t_{2}},B_{t_{2}}))$. This gives us a tiling of $t$ with even $k_{(s,t)}$ and $k_{(r,t)}$. The first of these facts follows from the fact that each segment from the tiling from $t_{2}$ has had its length exactly doubled. The latter observation then follows from $k_{s}|(k_{s}+k_{r})$ and the fact that $k_{(s,t)}$ is even. Just as in Lemma \ref{TightBound2}, as elements of the codomain of $f_{1}$ are tilings such that $k_{(s,t)}=1$ elements of the codomain of $f_{2}$ have even $k_{(s,t)}$, we know that the codomains of $f_{1}$ and $f_{2}$ are disjoint. 

We now give a piece-wise definition of $f_{p_{m}}$, with one definition when $k_{(s,t_{p_{m}})}$ of the input is even and another when $k_{(s,t_{p_{m}})}$ of the input is odd. For some, $(A_{t_{p_{m}}},B_{t_{p_{m}}})\in t_{p_{m}}$, if $k_{(s,t_{p_{m}})}$ is even, we begin by applying $f_{1}(t,(A_{t_{p_{m}}},B_{t_{p_{m}}}))=(A',B')$. This increases the size of $A_{t_{p_{m}}}+B_{t_{p_{m}}}$ by a multiplicative factor of $p_{m}$ and does not change the size of $A_{t_{p_{m}}}$. This would tile exactly the first $p_{m}$ fraction of points of $A_{t}+B_{t}\in t$ (i.e. $A'+B'=[|A_{t}+B_{t}|/p_{m}]$). Thus, we build our final $A''$ from $A'$ by letting $A''=A'+\{x\cdot p_{m}|x\in[0,p_{m}]\}$. This then gives us our desired output of $f_{p_{m}}$ on such inputs, which is $(A'',B')=(A_{t},B_{t})$. As $k_{(s,t_{p_{m}})}$ is even, so is $k_{(r,t_{p_{m}})}$. As the number of points added to the first rift is $p_{m}-1$ (which is even), we have that $k_{(r,t)}$ is even. This makes the codomain of $f_{p_{m}}$ disjoint from the codomain of $f'_{1}$ in this case. Further, it follows from the portion of the definition based on $f_{1}$ that $k_{(s,t)}=1$. This makes the codomain of $f_{p_{m}}$ disjoint from the codomain of $f'_{2}$ in this case.

Now for the case when $k_{(s,t_{p_{m}})}$ of the input to $f_{p_{m}}$ is odd. For some, $(A_{t_{p_{m}}},B_{t_{p_{m}}})\in t_{p_{m}}$, we begin by applying
\[f_{2}(t,(A_{t_{p_{m}}},B_{t_{p_{m}}}))=(A^{*},B^{*}).\]
This would tile exactly the first $p_{m}$ fraction of points of $A_{t}+B_{t}$ (i.e. $A^{*}+B^{*}=[|A_{t}+B_{t}|/p_{m}]$). Thus, we build our final $B^{**}$ from $B^{*}$ by letting $B^{**}=(B^{*}+\{x\cdot p_{m}|x\in[0,p_{m}]\})\cup \{x\cdot p_{m}|x\in[0,p_{m}]\}$ to get us our desired output of $f_{p_{m}}$ on such inputs of $(A^{*},B^{**})=(A_{t},B_{t})$. As $k_{(s,t_{p_{m}})}$ is odd and we are multiplying the size of the first segment by $p_{m}$ (which is also odd), we are left with an odd value for $k_{(s,t)}$. This makes the codomain of $f_{p_{m}}$ disjoint from the codomain of $f'_{2}$ in this case. Further, it follows from the first part of the mapping that $k_{(s,t)}>1$. This makes the codomain of $f_{p_{m}}$ disjoint from the codomain of $f'_{1}$ in this case. As we have shown the codomains of each function to be disjoint, the lemma follows.
\end{proof}
We note that the bound in Lemma \ref{TightBound2PartialLowerBound} is relatively tight as, if one where to substitute in 
\[t'_{p_{m}}=\mathcal{T}\bigg(\frac{\alpha}{p_{m}},\Big[p_{m}2^{k}\prod_{i=2}^{m-1}p_{i}^{2}\Big]\bigg)\]
in place of $t_{p_{m}}$, the inequality would be provably false (where the only difference between the two is a multiplicative factor of $p_{m}$ in the set to be tiled). We now prove that $\underset{\alpha}{\text{max}}[|\mathcal{T}(\alpha,[n])|]$ grows at a super-polynomial rate for infinitely many $n$.
\begin{lemma}\label{LowerBoundCalculation}
For all $c\in\mathbb{R}$ there exists an $m\in\mathbb{Z}^{+}$ such that, for $k\in\mathbb{Z}^{+}$ and $\alpha=2^{\lfloor k/2\rfloor}\prod_{i=2}^{m}p_{i}$, we have that 
\[\bigg|\mathcal{T}\bigg(\alpha,\Big[2^{k}\prod_{i=2}^{m}p_{i}^{2}\Big]\bigg)\bigg|=\omega(2^{ck}).\]
\end{lemma}
\begin{proof}
First, observe that we can now use Lemma \ref{2k9case} in conjunction with Lemma \ref{TightBound2PartialLowerBound} to get arbitrarily close to $2^{k}$. If we rework the statement and proof of Lemma \ref{2k9case}, multiplying $n$ by $5^{2}$ and $a$ by $5$, notice that we can repeat all of the arguments of the proof, except that the line 
\[|\mathcal{T}(2^{\lceil k/4\rceil},[2^{\lceil k/2\rceil}])|>\Omega(2^{0.58k/2})=\Omega(2^{0.29k})\]
is replaced by
\[|\mathcal{T}(2^{\lceil k/4\rceil}3,[2^{\lceil k/2\rceil}3^{2}])|>\Omega(2^{0.79k/2})>\Omega(2^{0.38k}).\]
This process of increasing the $m$ (i.e. the number of primes squared in $n$) can be repeated any number of times. Each time, the growth of the choose function is multiplied by the previous tiling growth rate with an additional $1/2$ multiplicative factor in the exponent. More specifically, suppose for some $m$ that the number of tilings is $2^{xk}$ for $x=1-y$ and $0<y<1$. Then for $m+1$, the number of sumset tilings would be $2^{(x/2+0.5)k}=2^{(1-(y/2))k}$. Thus, increasing the value of $m$ by $1$ increases $c$ by $(1-c)/2$ and the limit of $c$ as $m$ goes to infinity is $1$.

In the proof of Lemma \ref{2k9case}, we chose to count the number of tilings based on the middle column of the bottom row of the version of Pascal's triangle produced utilizing Lemma \ref{IE2k9}. In fact, we could have chosen the middle column of any row some fraction of the way down the triangle from which to count tilings from instead. Notice that the value of ${k/2w\choose k/4w}$ for $0<w\leq 1$ is roughly equal to $2^{kH(w/2)/2}$, where $H$ is the binary entropy function and 
\[H(w)=-w\log w-(1-w)\log(1-w).\]
Further, notice that $1-(w/2)$ is the amount we multiply $k$ by in $n$. Thus, adjusting $w$ increases the number of tilings of the $\mathcal{T}$ instance in question, but reduces the coefficient derived from the binomial coefficient. Similarly to the prior case of $w=1$, suppose for some $m$ that the number of tilings is $2^{xk}$ for $x=(H(w/2)/w)-y$ and $0<y<H(w/2)/w$. Then for $m+1$, the number of tilings would be
\[2^{(x(1-(w/2))+H(w/2)/2)k}=2^{((H(w/2)/w)-y+(w/2)y-H(w/2)/2+H(w/2)/2)k}=2^{(H(w/2)/w)+y((w/2)-1)k}.\]
Given this, the max value of $c$ we can achieve for a fixed $w$ is $\lim_{m\to\infty}c=H(w/2)/w$. Notice that $w$ decreases asymptotically faster than $H(w/2)$ and that, for sufficiently large $k$, we can choose an arbitrarily small $w$. Thus, the max value of $c$ tends to infinity as $w$ tends to $0$ and $k$ tends to infinity.
\end{proof}

We note, while one could expand upon the infinite families of values of $n$ for which Lemma \ref{LowerBoundCalculation} holds, one cannot prove a super-polynomial bound for all $n$. This follows immediately from the fact that, if $n$ is prime, the only tile sizes that tile $n$ are when $\alpha\in\{1,n\}$ and in each case, there is only a single distinct tiling of $[n]$.

\section{Extending to $\mathbb{Z}^{d}$}
We now extend the above results for finite subsets of $\mathbb{Z}$ to results for $\mathbb{Z}^{d}$ for arbitrary $d\in\mathbb{Z}^{+}$. In their conclusion, Bodini and Rivals \cite{BR06} correctly claim that their results about tiling finite intervals of the discrete line can be extended to $\mathbb{Z}^{d}$ and briefly mention what the characterization of the tilings in this setting would be. As we work with the additional restriction of fixed tile size we adjust our claim appropriately, but we note that our core idea appears to align witht heir own. Further, our proof verifies their claim, as one can sum over all tile sizes to recover their claim. Once we have extended our structural results to $\mathbb{Z}^{d}$, we use this to show that our upper and lower bounds from the previous section still hold in $\mathbb{Z}^{d}$.\footnote{The fact that the lower bound still holds follows immediately from the fact that one can set $d=1$ and use the same case and argument as in Lemma \ref{TightBound2PartialLowerBound} and Lemma \ref{LowerBoundCalculation}.} The extension of our structural results appears as Lemma \ref{d-dim} and proceeds by induction however, the base case of the induction is somewhat non-trivial in that it requires a fact about the distribution of the elements of $A$ with respect to $B$ for any $(A,B)$ that is valid tiling of $[n]$. For the sake of clarity, we separate out this base case structural result and present it as Lemma \ref{Overlap} before proving Lemma \ref{d-dim}.
\begin{lemma}\label{Overlap}
Let $C=[n]$. For $m\in\mathbb{Z}$, either $|A\cap(B+m)|=1$ or $|C\cap(B+m)|<|B+m|$.
\end{lemma}
\begin{proof}
We proceed by induction on the $\Omega^{*}(n)$. For the base case of $\Omega^{*}(n)=1$, we have that exactly one of $|A|$ and $|B|$ is $n$. Assume without loss of generality that $|A|=n$, then $A=C$ and $B+m$ is either in $C$ and has a unique intersection with $A$ or is not in $C$. Thus, the base case holds. Suppose the lemma holds for $\Omega^{*}(n)=k-1$ and we prove that it holds for $\Omega^{*}(n)=k$. Notice that, as $k>1$, either $A$ or $B$ has segments of size greater than $1$. Suppose without loss of generality that $A$ has segments of size greater than $1$. Thus, we can group $C$ into meta-points such that, for each meta-point, the consecutive elements it encompasses are either all in $A$ or all not in $A$. Further, by the definition of $B$, its segment size is $1$ and its minimum rift size is the segment size of $A$ minus $1$. Thus, it follows that each meta-point of $C$ is such that either its first element is in $B$ or it has no elements in $B$. Let $w>1$ be the number of points per meta-point and suppose $m$ is such that $|C\cap(B+m)|=|(B+m)|$ (if not, we are done). As meta-points have size greater then $1$, we can reconsider the intersection of $|A\cap(B+m)|=1$ with respect $n/w$ points and $\lfloor m/w\rfloor$ and apply our inductive hypothesis to say that exactly one meta-point of $C$ that is entirely elements of $A$ intersects a meta-point of $B+\lfloor m/w\rfloor$. This gives us an intersection of $A\cap(B+\lfloor m/w\rfloor)$ at the first point of a meta point. If we replace $\lfloor m/w\rfloor$ by $m$ as required, this increases the shift by at most $w-1$. As the intersection between $A$ and $B+\lfloor m/w\rfloor$ is in the first point of a meta-point and the meta-point in question contains $w$ consecutive elements, $|A\cap(B+m)|=1$ still holds and the lemma follows.
\end{proof}
We can now prove our main lemma of the section.
\begin{lemma}\label{d-dim}
Let $C=C_{1}\times C_{2}\times\ldots\times C_{d}$ for $C_{i}$ such that $C_{i}=[n_{i}]$ for some $n_{i}\in\mathbb{Z}^{+}$. We have that
\[|\mathcal{T}((\alpha_{1},\alpha_{2},\ldots,\alpha_{d}),C)|=|\mathcal{T}(\alpha_{1},C_{1})|\cdot|\mathcal{T}(\alpha_{2},C_{2})|\cdot\ldots\cdot|\mathcal{T}(\alpha_{d},C_{d})|.\]
More specifically, $(A,B)$ is an element of valid tilings of $\mathcal{T}((\alpha_{1},\alpha_{2},\ldots,\alpha_{d}),C)$ if and only if\\$A=\{(x_{1},x_{2},\ldots,x_{d}):x_{i}\in A_{i}\}$ and $B=\{(y_{1},y_{2},\ldots,y_{d}):y_{i}\in B_{i}\}$ where $(A_{i},B_{i})\in\mathcal{T}(\alpha_{i},C_{i})$. Further, for $m\in\mathbb{Z}^{d}$, either $|A\cap(B+m)|=1$ or $|C\cap(B+m)|<|B+m|$.
\end{lemma}  
\begin{proof}
To tile the elements of $C_{1}\times\text{min}[C_{2}]\times\ldots\times \text{min}[C_{d}]$, the only elements of $A$ that can be utilized are those in $C_{1}\times \text{min}[C_{2}]\times\ldots\times \text{min}[C_{d}]$. Thus, the elements of $A$ in $C_{1}\times \text{min}[C_{2}]\times\ldots\times \text{min}[C_{d}]$ and the elements of $B$ in $\mathbb{Z}^{+}\times \{0\}\times\ldots\times \{0\}$ exactly correspond to the one-dimensional tiling of $C_{1}$ with $|A|=\alpha_{1}$.

We prove the lemma via two nested inductive arguments. We begin with induction on the dimension $d$. For $d=1$, this follows by definition and Lemma \ref{Overlap}. Thus, suppose this is the case for some $d$ and we wish to prove that the lemma still holds for $d+1$. To do this, we prove by induction on $z$ that, for all $z\in[n_{d+1}]$ there exists a $k\in[z]$ such that, $C_{1}\times C_{2}\times\ldots\times C_{d}\times z$ is tiled by
\[(A\cap(C_{1}\times C_{2}\times\ldots\times C_{d}\times k))+\{b\in B:\text{proj}_{d+1}=z-k\}.\]
For the case of $z=1$ it follows from our inductive hypothesis relative to $d$ that $C_{1}\times\ldots\times C_{d}\times 1$ can only be tilings as described in the lemma. More specifically, one takes the $d$ dimensional tiling, then includes $1$ and $0$ as the $(d+1)^{\text{th}}$ element in the tuples of each element of $A$ and $B$ respectively. Now, for our inductive hypothesis relative to $z$, suppose the claim holds for the subsets of $A$ and $B$ that tile exactly the elements of $C_{1}\times C_{2}\times\ldots\times C_{d}\times [z-1]$ for $z\leq n_{d+1}$. We wish to prove this holds for $C_{1}\times C_{2}\times\ldots\times C_{d}\times z$.

Let $A'$ be the elements of $A$ in $C_{1}\times C_{2}\times\ldots\times C_{d}\times 1$. As $(1,\dots,1,z)$ must be tiled, this must either be because $(1,\dots,1,z)\in A$ or $(1,\dots,1,z)$ is covered by a translation of an element of $A$ in $C_{1}\times C_{2}\times\ldots\times C_{d}\times [z-1]$. First, suppose $(1,\dots,1,z)\in A$. Let $B^{*}\subset B$ be such that 
\[A'+B^{*}=C_{1}\times C_{2}\times\ldots\times C_{d}\times 1\]
and let $b'=(y_{1},\ldots,y_{d+1})$ for $y_{d+1}\in[z-1]$ be an element of $B$ that translates an element of $A\cap(C_{1}\times C_{2}\times\ldots\times C_{d}\times [z-1])$ to $C_{1}\times C_{2}\times\ldots\times C_{d}\times z$. For now, let us assume that $y_{d+1}=z-1$. By our inductive hypothesis on $d$, we have that either $|A\cap(B+m)|=1$ or $|C\cap(B+m)|<|B+m|$ for $m\in\mathbb{Z}^{d}$. Thus, it follows that $|((1,\dots,1,z)+B^{*})\cap(A'+b')|=1$ or $|C\cap A'+b'|<|A'+b'|$ which implies that $C_{1}\times C_{2}\times\ldots\times C_{d}\times z$ cannot be tiled by elements of $A'$ if it $(1,\dots,1,z)\in A$. Notice then that this together with our inductive hypothesis on $z$ implies that, for $k\in[z-1]$, $A\cap(C_{1}\times C_{2}\times\ldots\times C_{d}\times k)=\emptyset$, or that $A\cap(C_{1}\times C_{2}\times\ldots\times C_{d}\times k)=A'$. This observation then allows us to circle back and inductively remove the restriction that $y_{d+1}=z-1$ from the prior argument. Taken together, these restrictions turn the case of tiling $C_{1}\times C_{2}\times\ldots\times C_{d}\times z$ when $(1,\dots,1,z)\in A$ into tiling a $d$-dimensional space with fixed $B^{*}$, at which point we can apply the inductive hypothesis $d$ to conclude that the claim holds.

The argument for the case of $(1,\dots,1,z)\not\in A$ proceeds similarly, so we merely outline the differences at a high level. To cover $(1,\dots,1,z)$, it follows that there exists a $b'$ such that $(A\cap(C_{1}\times C_{2}\times\ldots\times C_{d}\times k))+b'=(A'+(0,\ldots,0,k-1))+b'$ covers $(1,\dots,1,z)$. Adding any element $a'\in C_{1}\times C_{2}\times\ldots\times C_{d}\times z$ to $A$ then creates an issue, as $a'=(1,\dots,1,z)+m$ for some $m$ such that $\text{proj}_{d+1}(m)=0$. At this point we treat $(A\cap(C_{1}\times C_{2}\times\ldots\times C_{d}\times k))+b'$ and $a'+B^{*}$ as $A$ and $B$ respectively from the statement that, for $m\in\mathbb{Z}^{d}$, either $|A\cap(B+m)|=1$ or $|C\cap(B+m)|<|B+m|$.
\end{proof}
Lemma \ref{d-dim} can then be leveraged to prove that $\underset{\alpha,C}{\text{max}}[|\mathcal{T}(\alpha,C)|]$ is maximal when $d=1$
\begin{lemma}\label{dcount}
Let $C=[x_{1}]\times[x_{2}]\times\ldots [x_{d}]$ for $x_{1},\ldots,x_{d}\in\mathbb{Z}^{+}$ and $C'=[n]$ for $n=|C|$. It follows that
\[\underset{\alpha,C}{\textup{max}}[|\mathcal{T}(\alpha,C)|]\leq \underset{\alpha',C'}{\textup{max}}[|\mathcal{T}(\alpha',C')|].\]
\end{lemma}
\begin{proof}
Recall that we use $[a]$ for the set of integers $1$ through $a$ and $[a,b]$ for the set of integers from $a$ to $b$ (inclusive of $a$ and $b$). We give a mapping from any $(A,B)\in\mathcal{T}(\alpha,C)$ to a distinct $(A',B')\in\mathcal{T}(\alpha',C')$. We know by Lemma \ref{d-dim} that $\prod_{i\in[d]}|\text{proj}_{i}(A)|=|A|$. We define $A'$ in terms of $A$ by specifying the elements of 
\[A'\cap\Big[\prod\limits_{i\in[k]}x_{i}\Big]\]
inductively with respect to $k$ with $k\leq d$. For $k=1$, we let $A'\cap[x_{1}]=\text{proj}_{i}(A)$. Suppose that the elements of
\[A'\cap\Big[\prod\limits_{i\in [k']}x_{i}\Big]\]
are defined for all $k'<k$ and that $1<k$. Then we define the elements of 
\[A'\cap\Big[\prod\limits_{i\in [k-1]}x_{i},\prod\limits_{i\in[k]}x_{i}\Big]\]
by
\[\bigcup\limits_{z\in\text{proj}_{k}(A)}\Bigg(\bigg(A'\cap\Big[\prod\limits_{i\in k-1}x_{i}\Big]\bigg)+(z-1)\Bigg).\]
We can prove by induction on $d$ that there exists a $B'$ such $(A',B')$. For $d=1$, we have that $C=C'$ and $(A,B)=(A',B')$. Suppose the lemma holds up to $d-1$. Then in the case of $d$, our definition of $A'$ from $A$ as well as the fact that
\[\prod\limits_{i\in[d]}|\text{proj}_{i}(A)|=|A|,\]
we have that $|A|=|A'|$. Consider $C'$ condensed into $x_{d}$ meta-points of size
\[\prod\limits_{i\in[d-1]}x_{i}\]
each. We know by the inductive hypothesis that there exists a $B^{*}\subset B'$ such that $A'+B^{*}$ tiles exactly each meta-point with at least one element of $A'$. This tiles exactly the meta-points indexed by elements of $\text{proj}_{d}(A)$. By Lemma \ref{d-dim}, $\text{proj}_{d}(A)$ is such that $(\text{proj}_{d}(A),B')\in\mathcal{T}(|\text{proj}_{d}(A)|,[x_{i}])$ for some $B'$ and thus, the lemma follows.
\end{proof}
Note that one cannot hope to achieve an improvement of Lemma \ref{dcount} that yields equality for all $\alpha$, $n$ and $d$, as evidenced by the following counterexample.
\begin{cor}\label{nodeq}
Let $C=[x_{1}]\times[x_{2}]\times\ldots [x_{d}]$ for $x_{1},\ldots,x_{d}\in\mathbb{Z}^{+}$ and $C'=[n]$ for $n=|C|$. $\exists \alpha,n,d\in\mathbb{Z}^{+}$ such that
\[\underset{\alpha,C}{\textup{max}}[|\mathcal{T}(\alpha,C)|]<\underset{\alpha',C'}{\textup{max}}[|\mathcal{T}(\alpha',C')|].\]
\end{cor}
\begin{proof}
Let $\alpha=2$ and $C=[3]\times[2]$. The only valid tiling of $C'$ is 
\[(\{(1,1),(2,1)\},\{(0,0),(0,1),(0,2)\})\]
where as $C'=[6]$ has the tilings $T_{1}=(\{1,2\},\{0,2,4\})$ and $T_{2}=(\{1,4\},\{0,1,2\})$.
\end{proof} 

\section{Tilings for Non-Contiguous $C$}

For all $\alpha$ and $n$, we hypothesis that $C=[n]$ has at least as many tilings as any $\alpha'$ and $C'$ such that $|C'|=n$. More formally, we claim the following
\begin{conj}\label{[n]isBest}
\[\forall \alpha,n,d\in\mathbb{Z}^{+}, \forall C\subset\mathbb{Z}^{d}\Big[(|C|=n)\implies|\mathcal{T}(\alpha,C)|\leq|\mathcal{T}(\alpha,[n])|\Big].\]
\end{conj}

While we are unable to resolve this conjecture, we make some progress towards this by proving that Conjecture \ref{[n]isBest} holds when $C$ can tile $[n]$ for some $n$. We begin by proving this only for the case of $d=1$ in Lemma \ref{TileTilings}, but we briefly show how to extend this to any $d$ in Corollary \ref{dTileTiling}. Before we proceed with Lemma \ref{TileTilings}, we make two useful structural claims. Corollary \ref{MinRift} clarifies that the first rift in any tiling is the smallest. Corollary \ref{InitialOrNone} is a structural result about valid tilings of $[n]$ that shows the segments only ever appear as the first $k_{s}$ elements of a meta-point (where meta-points are as defined in Definition \ref{meta-points}).

\begin{cor}\label{MinRift}
For any $(A,B)\in\mathcal{T}(\alpha,[n])$ with $j$ rifts we have that $\forall i\in[1,j](|r_{1}|\leq r_{i})$.
\end{cor}

\begin{proof}
Suppose not, then consider $B^{*}\subset B$ such that $B^{*}$ is the minimum size subset set of $B$ for which $[1,k_{s}+k_{r}]\subset A+B^{*}$ holds. We know from Lemma \ref{RecurUpperBound} that elements in $[1,k_{s}+k_{r}]$ must be covered by shifts of elements in $s_{1}$ by elements in $B^{*}$. Let $r_{i}$ be such that $|r_{i}|<|r_{1}|$, then $s_{i}+B^{*}\cap s_{i+1}\neq\emptyset$, which contradicts our assumption that $(A,B)\in\mathcal{T}(\alpha,[n])$.
\end{proof}

\begin{cor}\label{InitialOrNone}
For all $(A,B)\in\mathcal{T}(\alpha,[n])$, we have that 
\[\exists! M\subset\bigg[0,\frac{n}{(k_{r}/k_{s})+1}-1\bigg]\bigg(A=\bigcup\limits_{m\in M}([1,k_{s}]+\{m(k_{s}+k_{r})\})\bigg).\]
Put another way, there is a unique subset of the meta-points of $C$ with respect to $A$ such that the elements of $A$ are the first $k_{s}$ elements of these meta-points.
\end{cor}

\begin{proof}
Notice that, by the definition of $k_{s}$, it is certainly the case that $[1,k_{s}]\subset A$. As is shown in the arguments of Lemma \ref{LowerBound}, the remaining segments of $A$ outside of $s_{1}$ must begin at intervals of $k_{s}+k_{r}$. The upper bound on the interval of which $M$ is a subset then comes from the fact that an element of $A$ beyond this point would be translated by an element of $B$ to a value greater than $n$, making the tiling invalid.
\end{proof}
With this, we can now proceed with proving the main lemma of the section.
\begin{lemma}\label{TileTilings}  
For any $\alpha,\alpha', n\in\mathbb{Z}^{+}$ and any $(A,B)\in\mathcal{T}(\alpha,[n])$, we have that $|\mathcal{T}(\alpha',A)|\leq|\mathcal{T}(\alpha',[\alpha])|$.
\end{lemma}

\begin{proof}
We prove the lemma via the following steps:
\begin{enumerate}
\item We define a compression function $f$ for taking certain well structured sets of integers and outputting a new set of integers with larger groups of contiguous elements (i.e. larger segment size).
\item We define an injective function $g_{1}$ that maps any $(A,B)\in\mathcal{T}(\alpha,[n])$ such that $A$ has at least one rift, to some $(A',B')\in\mathcal{T}(\alpha,[n/(k_{r}/k_{s})+1)])$.
\item We define an injective function $g_{2}$ that maps any $(\mathcal{A},\mathcal{B})\in\mathcal{T}(\alpha',A)$ to some $(\mathcal{A}',\mathcal{B}')\in\mathcal{T}(\alpha',[A'])$.
\item We an injective function $g_{3}$ thats maps any $(\mathcal{A},\mathcal{B})\in(\mathcal{T}(\alpha',A)$ to an element of $\mathcal{T}(\alpha',[\alpha])$, where $g_{3}$ simply applies $g_{2}$ to $(\mathcal{A},\mathcal{B})$ a number of times based upon $g(A,B)$ for $(A,B)\in\mathcal{T}(\alpha,[n])$.
\end{enumerate}
Define the function $f(A,C)$, where there exists a $B$ such that $(A,B)\in\mathcal{T}(\alpha,C)$, to be such that
\[f(A,C)=\{a-k_{s}\lfloor a/(k_{s}+k_{r})\rfloor:a\in A)|\}.\]
Note that, while $C$ is not explicitly mentioned on the right hand side of the equality, it is implicit by the use $k_{s}$ and $k_{r}$ and thus, $C$ is in fact required in the input for the function to be well defined. Let $C$ be a finite contiguous set of integers. By corollary \ref{MinRift}, the intuition as to the effect of $f$ is that it takes the first $k_{s}$ elements of $A$ from each meta-point of $C$ with respect to $A$ and makes them contiguous by removing elements besides the first $k_{s}$ of each meta-point, followed by ``shifting" those remaining appropriately (which formalize in a moment).

Before defining $g_{1}$ and $g_{2}$, we define $\tilde{B}$ to be such that
\[\tilde{B}=\Big\{b_{i}\in B:i=1\text{ mod } \frac{k_{r}}{k_{s}}+1\Big\}.\]
By corollary \ref{InitialOrNone}, $B$ is defined in such a way that the elements of $A+\tilde{B}$ cover exactly the first $k_{s}$ elements of each meta-point of $[n]$ with respect to $A$. More formally, we have that 
\[(A,\tilde{B})\in\mathcal{T}\Bigg(\alpha,\bigcup_{m\in [0,\frac{n}{(k_{r}/k_{s})+1}-1]}([1,k_{s}]+\{m(k_{s}+k_{r})\})\Bigg).\]
We can now define both $g_{1}$ and $g_{2}$. Let $g_{1}(A,B)=(f(A,[n]),f(\tilde{B},[0,n-1])=(A',B')$. By Corollary \ref{InitialOrNone}, this transformation is such that $|A'|=|A|$ and $|B'|=|\tilde{B}|=k_{s}|B|/k_{r}$. This compression of the element of both $A$ and $\tilde{B}$ leads to $A'$ and $B'$ such that $(A',B')\in\mathcal{T}(\alpha,[n/(k_{r}/k_{s})+1)])$. Similarly we define $g'(\mathcal{A},\mathcal{B})=(f(\mathcal{A},[n]),f(\mathcal{B},[0,n-1])=(\mathcal{A}',\mathcal{B}')$. In the case of $g_{2}$, we now have that $|\mathcal{A}'|=|\mathcal{A}|$ and $|\mathcal{B}'|=|\mathcal{B}|$. Further, We have that $(\mathcal{A}',\mathcal{B}')\in\mathcal{T}(\alpha',A')$ where $\alpha'=|\mathcal{A}|$.

We now define $g_{3}$ from $g_{2}$, using $g_{1}$ to show its correctness. Suppose $A$ has $\ell$ distinct lengths of rift. Certainly, $\ell$ is finite and less than or equal to the total number of rifts in $A$. Notice though that applying $g_{1}$ to $(A,B)$ results in an $(A',B')$ with one less distinct rift length. Thus, it follows that $g_{1}^{\ell}(A,B)=([\alpha],\{0\})$. Knowing this, observe that we can then apply $g_{2}$ the same number of times (i.e. $\ell$ times) to $(\mathcal{A},\mathcal{B})\in\mathcal{T}(\alpha',A)$, giving us $g_{2}^{\ell}(\mathcal{A},\mathcal{B})=(\mathcal{A}',\mathcal{B}')$ such that $\mathcal{A}'+\mathcal{B}'=[\alpha]$. Thus, we define $g_{3}$ to be $g_{2}^{\ell}(\mathcal{A},\mathcal{B})$ where $\ell$ is the number of rifts in $\mathcal{A}+\mathcal{B}$. As $g_{2}$ is injective, it follows that distinct tilings $(\mathcal{A},\mathcal{B})\in\mathcal{T}(\alpha',A)$ map to distinct elements of $\mathcal{T}(\alpha',[\alpha])$ by $g_{3}$. Thus, the lemma follows.
\end{proof}
We now generalize Lemma \ref{TileTilings} to arbitrary dimension.
\begin{cor}\label{dTileTiling}
Let $C=C_{1}\times C_{2}\times\ldots\times C_{d}$ for $C_{i}$ such that $C_{i}=[n_{i}]$ for some $n_{i}\in\mathbb{Z}^{+}$. For all $(A,B)\in|\mathcal{T}((\alpha_{1},\alpha_{2},\ldots,\alpha_{d}),C)|$ and any set of $\alpha_{i}'\in\mathbb{Z}^{+}$ for all $i$, we have that 
\[|\mathcal{T}((\alpha'_{1},\alpha'_{2},\ldots,\alpha'_{d}),A)|\leq|\mathcal{T}((\alpha'_{1},\alpha'_{2},\ldots,\alpha'_{d}),[\alpha_{1}]\times[\alpha_{2}]\times\ldots\times[\alpha_{d})])|.\]
\end{cor}

\begin{proof}
Applying the function $g_{3}$ as in Lemma \ref{TileTilings} to each dimension one at a time injectivley maps any element of $\mathcal{T}((\alpha'_{1},\alpha'_{2},\ldots,\alpha'_{d}),A)$ to an element of $\mathcal{T}((\alpha'_{1},\alpha'_{2},\ldots,\alpha'_{d}),[\alpha_{1}]\times[\alpha_{2}]\times\ldots\times[\alpha_{d})])$.
\end{proof}

\section{Acknowledgments}
I want to thank Neng Huang for numerous discussions on the above results, during which he caught several errors and gave recommendations as to how to improve clarity. I also wish to thank David Cash and Andy Drucker for reading several drafts of this work and providing valuable feedback. Further, I wish to thank the anonymous reviewer who's excellent feedback helped me improve this work in several ways, including the clarification of this works connections to prior results on tilings.

\end{document}